\documentclass[11pt]{amsart}
\usepackage{amsmath, amsthm, amssymb}
\usepackage{graphicx}
\usepackage{mathbbol}
\usepackage{txfonts}
\usepackage[usenames]{color}
 \usepackage{ulem}

\numberwithin{equation}{section}

\newtheorem{thm}{Theorem}[section]
\newcommand{\bt}{\begin{thm}}
\newcommand{\et}{\end{thm}}

\newtheorem{cor}[thm]{Corollary}   

\newcommand{\bc}{\begin{cor}}
\newcommand{\ec}{\end{cor}}

\newtheorem{lem}[thm]{Lemma}   

\newcommand{\bl}{\begin{lem}}
\newcommand{\el}{\end{lem}}

\newtheorem{prop}[thm]{Proposition}
\newcommand{\bp}{\begin{prop}}
\newcommand{\ep}{\end{prop}}

\newtheorem{defn}[thm]{Definition}

\newcommand{\bd}{\begin{defn}}    
\newcommand{\ed}{\end{defn}}

\newtheorem{rmrk}[thm]{Remark}   

\newcommand{\br}{\begin{rmrk}}
\newcommand{\er}{\end{rmrk}}

\newtheorem{example}[thm]{Example}

\newcommand{\be}{\begin{equation}}
\newcommand{\ee}{\end{equation}}

\newcommand{\R}{\mathbb{R}}

\newcommand{\E}{\mathbb{E}}

\newcommand{\diam}{\operatorname{diam}}

\newcommand{\set}{\rm{set}}
\newcommand{\Ricci}{\rm{Ricci}}
\newcommand{\disjointunion}{\sqcup}

\newcommand{\mass}{{\mathbf M}}

\newcommand{\intcurr}{{\mathbf I}}      

\newcommand{\vol}{\operatorname{Vol}}

\begin{document}

\title[Smooth Convergence away from Singular Sets]{Smooth Convergence away
from Singular Sets}

\author{Sajjad Lakzian}
\thanks{Lakzian was partially supported as a doctoral student
by NSF DMS \#1006059.}
\address{CUNY Graduate Center}
\email{lakzians@gmail.com}

\author{Christina Sormani}
\thanks{Sormani was partially supported by NSF DMS \#1006059.}
\address{CUNY Graduate Center and Lehman College}
\email{sormanic@gmail.com}

\date{}


\begin{abstract}
We consider sequences of metrics, $g_j$, on a compact Riemannian
manifold, $M$, which converge smoothly on compact sets away from
a singular set $S\subset M$, to a metric, $g_\infty$, on $M\setminus S$.
We prove theorems which describe when $M_j=(M, g_j)$ converge
in the Gromov-Hausdorff sense to the metric completion,
$(M_\infty,d_\infty)$, of
$(M\setminus S, g_\infty)$.   To obtain these theorems, we study the
intrinsic flat limits of the sequences.   A new method, we call
hemispherical embedding, is applied to obtain explicit estimates on the
Gromov-Hausdorff and Intrinsic Flat distances between Riemannian
manifolds with diffeomorphic subdomains. 

\textcolor{blue}{Seven years after the publication of this paper in CAG, Brian Allen  
discovered a counter example to the published statement of Theorem 1.3.  Note that
Theorem 4.6 (which is the key theorem cited in other papers) remains correct.  We have
added an hypothesis to correct the statement of Theorem 1.3 and its consequences. This 
v4 includes corrections in blue, an erratum at the end of the introduction, and Brian Allen's example
in an appendix. An erratum is also being sent to the journal.}

\end{abstract}

\maketitle

\section{Introduction}

The purpose of this paper is to provide general results concerning
the limits of Riemannian manifolds which converge smoothly away from
a singular set as follows:

\begin{defn} \label{defn-smoothly-1}
We will say that a sequence of Riemannian metrics, $g_j$, on a
compact manifold, $M$, converges smoothly away from $S \subset M$
to a Riemannian metric $g_\infty$ on $M \setminus S$ if
for every compact set $K\subset M\setminus S$, $g_j$
converge $C^{k,\alpha}$ smoothly to $g_\infty$ as tensors.  \textcolor{blue}{In addition
we say that it converges uniformly from below if there exists $\delta_j\to 0$ such that
$g_j \ge (1-\delta_j)^2 g_\infty$ on $M\setminus S$.}   
\end{defn}

The techniques developed in this paper will also be applied to
other notions of smooth convergence away from singular sets in
upcoming work of the first author, particular notions in which the sequence
of manifolds need not be diffeomorphic.
With any notion of smooth convergence away from
a singular set, one must keep in mind that
even when the singular set is an isolated point, smooth convergence
away from that point does not even imply that $(M_\infty,g_\infty)$ is
compact [Example~\ref{ex-not-bounded}].  Increasingly large distances
may exist outside the compact sets used to define the smooth convergence.

Given two compact Riemannian manifolds, $M_i$, the Gromov-Hausdorff
distance, $d_{GH}(M_1, M_2)$, is an isometry invariant.  Introduced
by Gromov in \cite{Gromov-metric}, it is a distance on compact metric
spaces in the sense that $d_{GH}(M_1,M_2)=0$ iff $M_1$ is
isometric to $M_2$.   
When studying precompact domains within manifolds,
one always takes the metric completion before examining the region
using the Gromov-Hausdorff distance.
  Section~\ref{Sect-review} (see Definition~\ref{defn-GH}).   

Smooth limits away from singular sets, depend on the charts and
tensors $g_j$ used to define the smooth limit (c.f. Example~\ref{ex-to-hemisphere}).
Thus it is important to understand when the metric completion,
$\bar{Y}$, of a smooth limit, $Y=(M\setminus S, g_\infty)$, is
in fact actually the Gromov-Hausdorff limit, $(M_0,d_0)$, 
of the original sequence of manifolds, $(M_j, d_j)$, where $d_j$
is the Riemannian distance defined by the Riemannian metric $g_j$..
Observe that these spaces need not be isometric (c.f. Example~\ref{ex-region})
and that the original sequence of manifolds might not even have a
Gromov-Hausdorff limit (c.f. Example~\ref{ex-no-GH}).   If $M\setminus S$ is
not connected there isn't even a notion of the metric completion
as a single metric space (c.f. Example~\ref{ex-not-connected}).

Theorems relating Gromov-Hausdorff
limits and smooth limits away from singular sets appear in
work of Anderson,  Bando-Kasue-Nakajima,
Eyssidieux-Guedj-Zeriahi, Huang, Ruan-Zhong, Sesum, Tian and Tosatti
particularly in the setting of
Kahler Einstein manifolds
\cite{Anderson-KE} \cite{BKN}
\cite{EGZ} \cite{Huang-convergence}
\cite{Ruan-Zhang} \cite{Sesum-convergence}
\cite{Tian-surfaces} \cite{Tosatti-KE}.
However, even in this setting, the
relationship is not completely clear and the limits need not
agree \cite{Bando-bubbling}.

In this paper, our primary goal is to examine when the metric
completion, $(M_\infty, d_\infty)$, of the smooth limit,
$(M \setminus S, g_\infty)$, is isometric to the Gromov-Hausdorff
limit, $(M_0,d_0)$, of the original sequence of Riemannian manifolds
$(M,g_j)$.  We prove a number of theorems and present a
number of examples considering manifolds with and without Ricci curvature 
bounds.  Perhaps the most important result is the following:

\begin{thm}\label{Ricci-codim-thm}
Let $M_i=(M,g_i)$ be a sequence of oriented compact Riemannian manifolds
with uniform lower Ricci curvature bounds, 
\be
\Ricci_{g_i}(V,V)\ge (n-1)H \, g_i(V,V) \qquad \forall V \in TM_i
\ee
which converges smoothly away from $S$ \textcolor{blue}{uniformly from below}
where $S$ is a submanifold of codimension $2$.

If there is a connected precompact exhaustion, $W_j$, of
$M\setminus S$,
\be \label{defn-precompact-exhaustion}
\bar{W}_j \subset W_{j+1} \textrm{ with } 
\bigcup_{j=1}^\infty W_j=M\setminus S
\ee
satisfying 
\be \label{diam-2}
\diam(M_i) \le D_0,
\ee
\be \label{area-2}
\vol_{g_i}(\partial W_j) \le A_0,
\ee
and
\be \label{not-vol-2}
\vol_{g_i}(M\setminus W_j) \le V_j \textrm{ where } \lim_{j\to\infty}V_j=0,
\ee
then
\be
\lim_{j\to \infty} d_{GH}(M_j, N)=0,
\ee
where $N$ is the
metric completion of $(M\setminus S, g_\infty)$. 
\end{thm}

Note that, unlike prior existing results concerning the Gromov-Hausdorff
limits of manifolds, here we require only area and volume
controls on the connected precompact exhaustion.  
Theorem~\ref{Ricci-codim-thm} is a consequence
of Theorem~\ref{Ricci-smooth-to-GH}, stated within, 
which assumes only that the connected precompact exhaustion
is uniformly well embedded in the sense of Definition~\ref{well-embedded} \footnote{\textcolor{blue}{
Definition~\ref{well-embedded} of uniform well embeddedness has been revised fixing the order of the limits.}}.
The necessity of the various
hypothesis of these theorems is described in 
Remark~\ref{Ricci-necessity}.  In particular the diameter hypothesis is unnecessary when the Ricci curvature is nonnegative.

The Ricci curvature condition in these theorems may be replaced by
a requirement that the sequence of manifolds have a uniform linear contractibility function.  See Definition~\ref{defn-contractibility-function}, Theorem~\ref{c-smooth-to-GH} and Theorem~\ref{c-codim-thm}, stated within.
The necessity of the various
hypothesis of these theorems is described in 
Remark~\ref{c-necessity}.

Observe our main theorems concern
sequences of manifolds converging smoothly away from a singular
set satisfying (\ref{area-2}) and (\ref{not-vol-2}).   In order to control the
limits of such manifolds using only conditions on volumes, 
we apply techniques developed by the 
second author with Stefan Wenger in \cite{SorWen1} and
\cite{SorWen2}.
In attempt to keep this article self contained, we review convergence
of Riemannian manifolds in Section~\ref{Sect-review}.   We provide
extensive examples in Section~\ref{Sect-Examples}.   
All examples are proven in
detail with short statements for easy reference.

Our theorems are proven by studying 
the {\em intrinsic flat} limit of the manifolds
[Definition~\ref{dsw}].
This intrinsic flat distance,
$d_{\mathcal{F}}(M_1, M_2)$ was
originally defined in work of the
second author with Wenger \cite{SorWen2}.
It is estimated by explicitly constructing a filling manifold, $B^{m+1}$,
between the two given manifolds, finding the excess boundary
manifold $A^m$ satisfying (\ref{Stokes}) and summing their volumes
as in (\ref{est-int-flat}).   See Remark~\ref{rmrk-Z-c} for a straight forward
construction.  
Since $d_\mathcal{F}$ depends only on the Riemannian manifolds,
$M_i$, as oriented metric spaces
with a notion of integration over $m$ forms, we take settled completions
rather than metric completions of open domains when analyzing the
intrinsic flat distance (see Definition~\ref{defn-positive-density}).  
If two completely settled oriented Riemannian manifolds, $M_1$ and $M_2$
have $d_{\mathcal{F}}(M_1, M_2)=0$ then there is an orientation
preserving isometry between them \cite{SorWen2}.   
See Section~\ref{Sect-review} for a review of the intrinsic flat
distance and related concepts. 

In Section~\ref{Sect-estimates} we prove
new explicit estimates on the Gromov-Hausdorff,  intrinsic flat and
scalable intrinsic flat distances between pairs of manifolds which
are diffeomorphic on subdomains [Theorem~\ref{thm-subdiffeo}].\footnote{\textcolor{blue}{Theorem~\ref{thm-subdiffeo} is correct as originally stated and proven. }}
The subdomains need not be connected.   These estimates are found
by isometrically embedding the regions into a 
common metric space defined using a hemispherical construction [Proposition~\ref{prop-hem}]
and then measuring the Hausdorff, flat and scalable flat distances
between their images respectively
[Lemma~\ref{lem-squeeze-Z}].   Note that
the Hausdorff distance measures distances between the images
using tubular neighborhoods while the flat distance measures a 
filling volume between the images.  These estimates have been
applied in work of the second author with Dan Lee on questions
concerning the Riemannian Penrose Inequality \cite{LeeSormani2}
and in the first author's doctoral dissertation \cite{Lakzian-thesis}.

In Section~\ref{sect-IF}, we prove theorems concerning the
intrinsic flat limits of manifolds which converge smoothly away
from singular sets.   In particular, we prove:

\begin{thm}\label{codim-thm}
Let $M_i=(M,g_i)$ be a sequence of compact oriented Riemannian manifolds
such that
there is a submanifold, $S$, of codimension 2,
and connected precompact exhaustion,
$W_j$, of $M\setminus S$ satisfying (\ref{defn-precompact-exhaustion})
with $g_i$ converge smoothly to $g_\infty$ on $M\setminus S$
\textcolor{blue}{uniformly from below} such that
\be\label{m-diam}
\diam_{M_i}(W_j) \le D_0 \qquad \forall i\ge j, 
\ee
\be \label{m-area}
\vol_{g_i}(\partial W_j) \le A_0,
\ee
and
\be \label{m-edge-volume}
\vol_{g_i}(M\setminus W_j) \le V_j \textrm{ where } \lim_{j\to\infty}V_j=0.
\ee
Then
\be
\lim_{j\to \infty} d_{\mathcal{F}}(M_j', N')=0.
\ee
where  $N'$ is the settled completion
of $(M\setminus S, g_\infty)$. 
 \footnote{\textcolor{blue}{The appendix has Brian Allen's counter example
to the original statement of Theorem 1.3 without the uniform convergence from below.}}
\end{thm}

This theorem is a consequence of 
Theorem~\ref{flat-to-settled} which assumes only that 
the connected precompact exhaustion is uniformly well embedded  in the 
sense of Definition~\ref{well-embedded}.\footnote{\textcolor{blue}{As mentioned above, Definition~\ref{well-embedded}
of uniform well embeddedness has been corrected within.}} 
 We discuss
the necessities of the conditions for these theorems
in Remark~\ref{flat-necessity}.
A key step in the proof is
a technical proposition concerning the convergence of exhaustions
of manifolds [Proposition~\ref{flat-exhaustion}].  \footnote{ \textcolor{blue}{At the end of
Section 5 there was a Lemma~\ref{codim-2-lambda} concerning uniform well embeddedness which has
a new proof using the uniform convergence from below.}}

In Section~\ref{Sect-flat-to-GH}, we apply the theorems regarding
intrinsic flat limits to prove the theorems concerning Gromov-Hausdorff
limits mentioned earlier.    Note that the second author and
Stefan Wenger have proven that
the intrinsic flat and Gromov-Hausdorff limits
of sequences of manifolds 
agree when the sequence has
nonnegative Ricci curvature and the volume is bounded below uniformly
\cite{SorWen1}.    These results are reviewed in Section~\ref{Sect-flat-to-GH}.   Theorem~\ref{codim-thm} then immediately
implies Theorem~\ref{c-codim-thm} and Theorem~\ref{Ricci-codim-thm}
when $H=0$.\footnote{\textcolor{blue}{These theorems now require uniform convergence from below.} } To obtain Theorem~\ref{Ricci-codim-thm} for
arbitrary values of $H$,  we prove Proposition~\ref{prop-improve-sw}.

Applications of these results appear in 
joint work of the second author and Dan Lee concerning asymptotically
flat rotationally symmetric Riemannian manifolds with positive
scalar curvature that satisfy an almost equality in the Penrose inequality
\cite{LeeSormani2}.   We believe these results may also be
applicable to open questions stated in \cite{LeeSormani1}.  
The first author is examining further applications in his doctoral dissertation.

The authors would like to thank the Simons Center for Geometry and Physics
for its hospitality.  Attending the many interesting talks there made it clear that
a paper clarifying the applications of the intrinsic flat convergence to
understand smooth limits away from singular sets would be useful to
mathematicians in a wide variety of subfields of geometric analysis.  We would
also like to thank Xiaochun Rong for suggesting an important counter example
and the referee for providing thorough and detailed suggestions that 
improved this paper throughout. 

\vspace{.2cm}

\noindent\textcolor{blue}{{\bf Errata:}}

\textcolor{blue}{Seven years after this paper was published in CAG, Brian Allen found a counter
example to Theorem 1.3 which we present in the appendix.  
Theorem 1.3 is false as stated in the original publication for smooth convergence $g_j\to g_\infty$
on $M\setminus S$ where the converge is only uniform on compact sets $K \subset M\setminus S$.  
This theorem and its consequences
Theorem 1.2 and Theorem~\ref{c-codim-thm} are easily corrected by adding in the 
assumption that the convergence of $g_j\to g_\infty$ is also uniform from below on $M\setminus S$.}

\textcolor{blue}{ The original error was traced by Brian Allen to a reversal of indices in limits in the original proof of
Theorem 5.2.    We find that by correcting the order of the
limits in Definition 5.1 of uniform well embeddedness, we can prove Theorem 5.2
as originally stated.   
We also correct the proof Lemma 5.7 to adapt to this new definition of
uniform well embeddedness using the notion of smooth convergence away from a singular set 
uniformly from below.  Thus Theorem 1.3 and its consequences (Theorem 1.2 and Theorem 6.6)
are true assuming this stronger hypothesis.  These corrections are made in blue within this v4 of the paper.}

\textcolor{blue}{ This paper has been cited many times since its publication.  We believe the only paper that 
needs revision is \cite{Lakzian-Diameter} by the first author of this paper.   The other papers apply only
Theorem~\ref{thm-subdiffeo}, which remains correct as originally stated and proven.}

\section{Background} \label{Sect-review}

All notions of distances between Riemannian manifolds studied
in this paper are built upon Gromov's idea that one may 
view Riemannian manifolds as metric spaces and isometrically
embed them into a common metric space.
In this paper, a key part of our work relies on constructing such
isometric embeddings.  We review Gromov's key
ideas in Subsection~\ref{subsect-iso}.

To estimate the Gromov-Hausdorff distance between a pair of
Riemannian manifolds, one needs only find a pair of isometric
embeddings $\varphi_i: M^m_i \to Z$ into a common complete
metric space $Z$ and then measure the Hausdorff distances
between the images.   
We review the definition of the Hausdorff and
Gromov-Hausdorff in
Subsection~\ref{subSect-flat-to-GH}.  

To estimate the intrinsic flat distance 
one must measure the  flat distance between these images.  
So one may construct a Riemannian manifold of one dimension
higher filling in the space between the two images, possible with
some excess boundary.   
Note that one can only measure 
the intrinsic flat distance between oriented manifolds with finite volume
of the same dimension.  
See Subsection~\ref{subsect-int-flat}.  

The scalable intrinsic flat distance is also defined using
filling manifolds and excess boundaries. 
It is reviewed in Subsection~\ref{subsect-scalable}.

Remarks~\ref{rmrk-Z-a}, ~\ref{rmrk-Z-b}, ~\ref{rmrk-Z-c}
and~\ref{rmrk-Z-s} capture the key properties of these three
notions of distance needed to estimate them for the purposes of
this paper.

\subsection{Metric Spaces and Isometric Embeddings}\label{subsect-isom-emb} \label{subsect-iso}

\begin{defn}\label{defn-d}
Recall that one may view a Riemannian
manifold $(M,g)$ as a metric space $(M,d)$ by defining the distances between
points as follows:
\be \label{g-to-d}
d(x_1, x_2)= \inf\left\{L_g(\gamma): \, \gamma(0)=x_1, \, \gamma(1)=x_2\right\}
\ee
where
\be \label{g-to-d-2}
L_g(\gamma)= \int_0^1 g(\gamma'(t), \gamma'(t))^{1/2} dt
\ee
Given a connected
subdomain, $W \subset M$, and $x,y\in W$, the "restricted metric",
$d_M(x,y)$, will
denote the distance between $x$ and $y$ measured as in 
(\ref{g-to-d}) where the infimum taken over 
all curves $\gamma: [0,1]\to M$, while the "induced length
metric", $d_W(x,y)\ge d_M(x,y)$, has the infimum taken only over curves
$\gamma: [0,1]\to W$.    We denote the
restricted and intrinsic length diameters of $U\subset W\subset M$
as follows
\begin{eqnarray}
\diam_M(U)&=&\sup\{d_M(x,y): \, x,y\in U\}\\
\diam_W(U)&=&\sup\{d_W(x,y): \, x,y\in U\}
\end{eqnarray}
\end{defn}

More generally a length metric space is a metric space whose
distances are defined as an infimum of lengths of rectifiable
curves.  Compact length metric spaces always have
minimizing geodesics between points achieving the distance.

In this paper we will often define metric spaces, $Z$, by gluing together
Riemannian manifolds with corners along their boundaries.  In this
way we may still apply (\ref{g-to-d}) to define the distances between
points.   Again, for connected subdomains, $W\subset Z$, one has 
both an induced length metric, $d_W$, and a restricted distance
$d_Z\le d_W$ just as in Definition~\ref{defn-d}.

\begin{defn} \label{defn-isom-embed}
An isometric embedding $\varphi: X \to Z$ is a distance preserving map
\be
d_Z(\varphi(x_1), \varphi(x_2)) = d_X(x_1, x_2) \qquad \forall x_1, x_2 \in X
\ee
\end{defn}

One should be aware that a Riemannian isometric embedding defined by
the fact that $d\varphi$ is an isometry on the tangent spaces at each point,
is not necessarily an isometric embedding.  For example, the natural embedding
of the sphere into Euclidean space is not an isometric embedding with the
standard metric on the sphere.   See Figure~\ref{fig-defn-isom-embed}.

\begin{figure}[h] 
   \centering
   \includegraphics[width=5in]{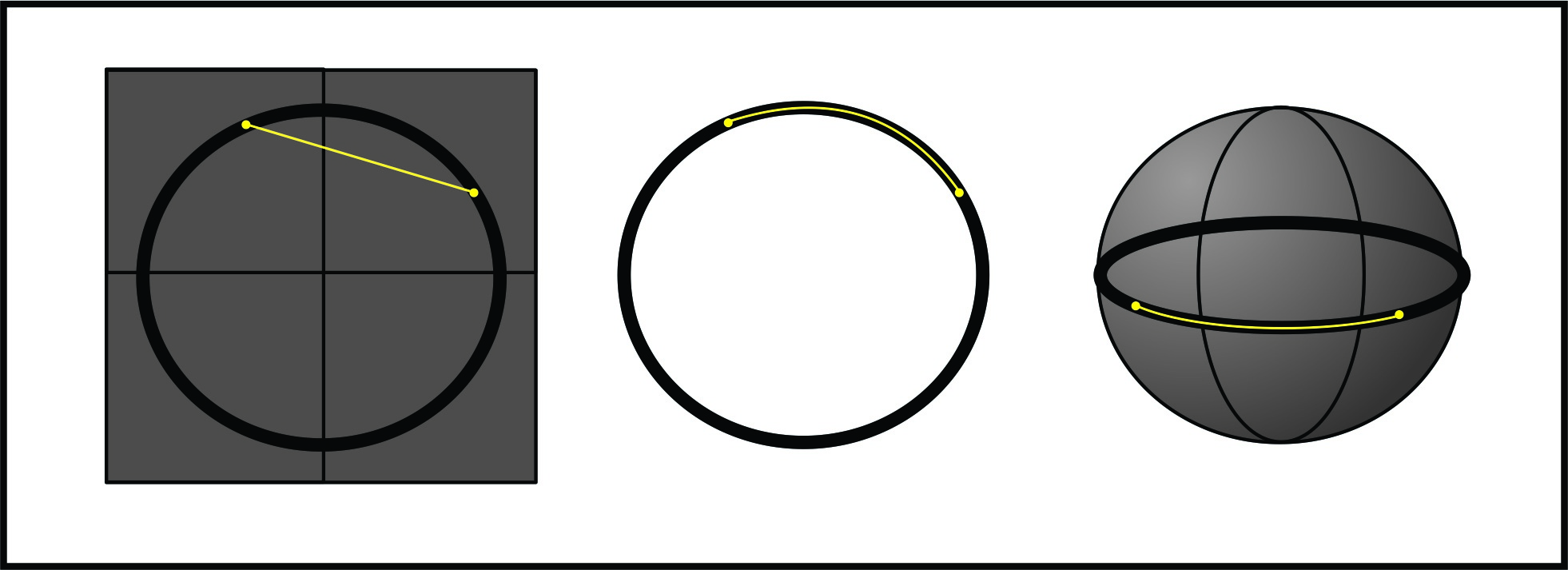}
   \caption{$S^1$ in the center isometrically embeds into $S^2$ on the right, but does not
isometrically embed into $\E^2$ on the left.  }
   \label{fig-defn-isom-embed}
\end{figure}

\begin{rmrk}\label{rmrk-Z-a}
Suppose two manifolds, $M_i$ have
diffeomorphic subdomains, $U_i$, then a filling manifold
can be constructed of the form $U\times [h_1,h_2]$ with a well
chosen metric $g'$ so that $M_i$ isometrically embed into
\be
Z=M_1 \disjointunion (U\times [h_1,h_2]) \disjointunion M_2.
\ee
Here $Z$ is glued together so that $U_i$ is identified point to point with
$U\times \{h_i\}$.  A precise way of
choosing such a $g'$ will be given in Theorem~\ref{thm-subdiffeo}.
See Figure~\ref{fig-defn-GH}.
\end{rmrk}

\subsection{The Gromov-Hausdorff Distance}\label{subSect-flat-to-GH}

The Gromov-Hausdorff distance between a pair of Riemannian
manifolds is estimated by taking isometric
embeddings into a common metric space $Z$
and measuring the Hausdorff distance between them.  This
distance was introduced by Gromov in \cite{Gromov-metric}.
It is defined on pairs of metric spaces.

\begin{defn} [Hausdorff] \label{defn-H}
Given two subsets $Y_1, Y_2 \subset Z$, the Hausdorff distance is
defined
\be
d_H^Z(Y_1, Y_2) = \inf\left\{r: \, Y_1 \subset T_r(Y_2) \textrm{ and } Y_2 \subset T_r(Y_1)\right\}
\ee
where $T_r(Y)=\left\{z\in Z: \, \exists y\in Y \, s.t. \, d(y,z)<r\right\}$.
\end{defn}

One may immediately observe that the topology and dimension of
subsets which are close in the Hausdorff sense can be quite different.

\begin{defn}[Gromov] \label{defn-GH}
Given a pair of metric spaces $(X_1, d_1)$ and $(X_2, d_2)$,
the Gromov-Hausdorff distance between them is
\be
d_{GH}(X_1, X_2)= \inf\left\{ d^Z_H(\varphi_1(X_1), \varphi_2(X_2)): \,\,\varphi_i: X_i \to Z\right\}
\ee
where the infimum is taken over all common metric spaces, $Z$,
and all isometric embeddings, $\varphi_i: X_i \to Z$.
\end{defn}

\begin{figure}[h] 
    \centering
    \includegraphics[width=5in]{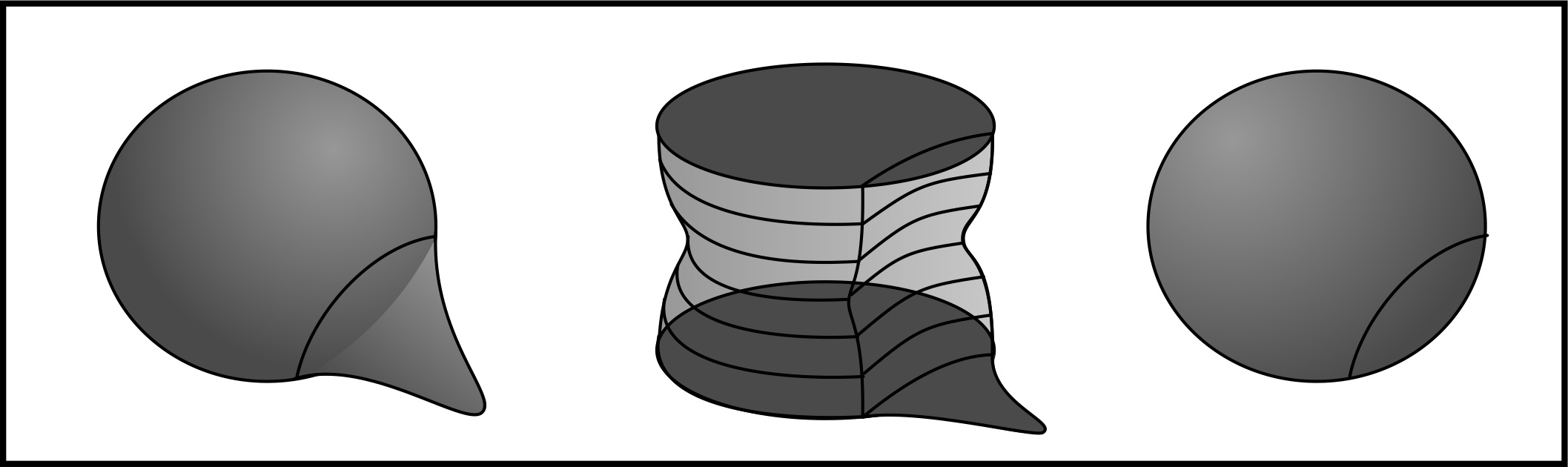}
    \caption{$M_1$ and $M_2$ depicted on the left and the right
    isometrically embed into $Z$ in the center.
    See Remarks~\ref{rmrk-Z-a} and~\ref{rmrk-Z-b}.}
        \label{fig-defn-GH}
  \end{figure}

\begin{rmrk}\label{rmrk-Z-b}
In Figure~\ref{fig-defn-GH} depicting Remark~\ref{rmrk-Z-a},
 we see that
\be
d_{GH}(M_1, M_2) \le d_H^Z(\varphi_1(M_1), \varphi_2(M_2))
\ee
which is roughly the length of a curve from the tip of the bump in $M_1$
running back within $\varphi(M_1)\subset Z$ to the warped region
$U\times [h_1,h_2]$ and then straight up to $M_2$.  Later in this paper
Theorem~\ref{thm-subdiffeo} we will find a precise description of
the metric on the metric space $Z$ of Figure~\ref{fig-defn-GH}.
\end{rmrk}

Gromov proved in \cite{Gromov-metric}
that this a distance between compact
metric spaces, in the sense
that $d_{GH}(X_1, X_2)=0$ iff $X_1$ and $X_2$ are isometric.
In general one takes the metric completion, $\bar{X}$, of a
precompact space, $X$, before
discussing it's Gromov-Hausdorff distance and we will do the same here.
Recall that the metric completion is defined as follows:

\begin{defn}\label{metric-completion}
Given a precompact metric space, $X$, the metric completion, $\bar{X}$
of $X$ is the space of Cauchy sequences, $\{x_j\}$, in $X$ with the
metric
\be
d(\{x_j\},\{y_j\})=\lim_{j\to\infty} d_X(x_j,y_j)
\ee
and where two Cauchy sequences are identified if the distance between
them is $0$.  There is an isometric embedding, $\varphi: X \to \bar{X}$,
defined by $\varphi(x)=\{x\}$ where $\{x\}$ is a constant sequence.
Lipschitz functions, $F: X\to Y$, extend to $F:\bar{X} \to Y$ via
$F(\{x_j\})=\lim_{j\to\infty} F(x_j)$ as long as $Y$ is complete.
\end{defn}

Gromov's compactness theorem states that a sequence of Riemannian manifolds $M_j^m$
with a uniform lower bound on Ricci curvature have a subsequence which converges
in the Gromov-Hausdorff sense.   More generally one may replace the Ricci curvature
bound with a bound on the number, $N(r)$, of 
disjoint balls of radius, $r$, that can be placed in
a metric space, $X$.   That is, a sequence of metric spaces $X_j$ with a uniform
bound on $N(r)$ for all $r$ sufficiently small, has a subsequence which converges in the
Gromov-Hausdorff sense to a compact limit space $X$.   Conversely, if
$d_{GH}(X_j , X)\to 0$, then there is a uniform bound on $N(r)$.
\cite{Gromov-metric}.

\subsection{The Intrinsic Flat Distance} \label{subsect-int-flat}

To estimate the intrinsic flat distance between a pair
of oriented Riemannian manifolds one again needs only find a pair of
isometric embeddings, $\varphi_i: M^m_i \to Z$, into a common complete
metric space, $Z$.  When one finds a 
filling submanifold, $B^{m+1}\subset Z,$ and an
excess boundary submanifold, $A^m\subset Z$, such that
\be\label{Stokes}
\int_{\varphi_1(M_1)}\omega -\int_{\varphi_2(M_2)}\omega=\int_B d\omega +\int_A \omega,
\ee
then the intrinsic flat distance is bounded by
\be \label{est-int-flat}
d_{\mathcal{F}}(M^m_1, M^m_2) \le \vol_m(A^m)+\vol_{m+1}(B^{m+1}). 
\ee
Generally the filling manifold and excess boundary can have corners and
more than one connected component.  See Figure~\ref{fig-subdiffeo-1}.

\begin{rmrk}\label{rmrk-Z-c}
In Figure~\ref{fig-defn-GH} depicting Remark~\ref{rmrk-Z-a}, 
we have $M_i$ isometrically embedded
into a well chosen metric space
\be
Z=M_1 \disjointunion (U\times [h_1,h_2]) \disjointunion M_2.
\ee
Applying (\ref{Stokes}) to
\be
B=U\times [h_1,h_2]
\ee
we see that the excess boundary 
\be
A= (M_1\setminus U_1) \cup (\partial U \times [h_1,h_2])
\cup (M_2 \setminus U_2).
\ee
Then
\begin{eqnarray*} 
d_{\mathcal{F}}(M_1, M_2) &\le& 
 \vol_m(M_1\setminus U_1) + \vol_m(M_2\setminus U_2)
+ \vol_m(\partial U \times [h_1,h_2])\\
&+&\vol_{m+1}(U \times [h_1,h_2],g').
\end{eqnarray*}
An explicit construction of the metric $g'$ on $U\times[h_1,h_2]$
in Theorem~\ref{thm-subdiffeo}, allows one to precisely
estimate the volume of $U\times[h_1,h_2]$
and $\partial U\times[h_1,h_2]$.
\end{rmrk}

To understand limits of sequences of Riemannian manifolds, 
the intrinsic flat distance was defined on a larger class of metric spaces
called integral current spaces in
\cite{SorWen2}.   An integral current space, $(X,d,T)$, is
a metric space, $X$, with a metric, $d$, and an integral current
structure, $T$, such that $\set(T)=X$.  
An oriented Riemannian manifold, $(M,g)$, of
finite volume, has a metric, $d_M$, defined as in Definition~\ref{defn-d}
and an integral current structure, $T$, acting on $m$ dimensional
differential forms, $\omega$ as
\be
T(\omega)= \int_M \omega.
\ee
More generally, the integral current structure, $T$, of an integral current
space, $(X,d,T)$, is an $m$ dimensional integral current
$T\in \intcurr_m(\bar{X})$ defined as in Ambrosio-Kirchheim's work \cite{AK}.   
The integral current structure $T$ provides both an orientation and
a measure called the mass measure denoted $||T||$ and 
$\set(T)$ is the set of positive lower density for this measure.  
On a oriented
Riemannian manifold, the mass measure is just the Lebesgue measure.
More generally the mass measure can have integer valued weights.

If $(M^m, g)$ is a Riemannian manifold with singularities on a subset $S$
such that the Hausdorff
measure, $\mathcal H_m(S)=0$, then one obtains a corresponding integral
current space by taking the settled completion of $M\setminus S$ defined as follows:

\begin{defn} \cite{SorWen2} \label{defn-positive-density}
The settled completion, $X'$, of a metric space $X$ with a measure $\mu$
is the collection of points $p$ in the metric completion $\bar X$
which have positive  lower density:
\be\label{eq-positive-density}
\liminf_{r\to 0} \mu(B_p(r))/r^m >0.
\ee
The resulting space is then "completely settled". 
\end{defn}

\begin{figure}[h] 
   \centering
   \includegraphics[width=5in]{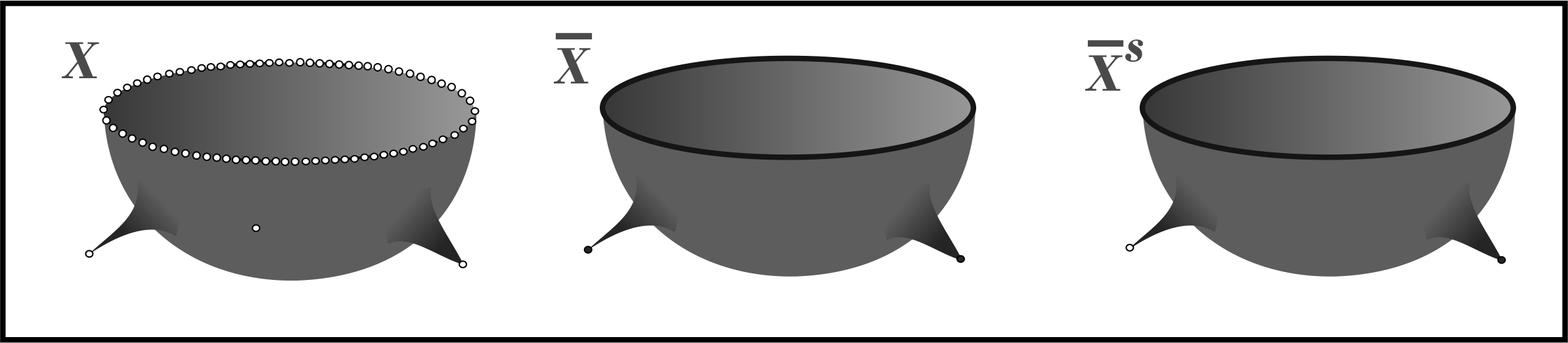}
   \caption{The completion, $\bar{X}$ includes the boundary and
fills in the three "holes" and the settled completion, $X'$,
removes the cusped
singularity but keeps the boundary, cone tip and smoothly filled hole. }
   \label{fig-settled-completion}
\end{figure}

If a manifold has only point singularities, one includes
all conical tips and does not include cusp tips in a manifold with point singularities.   See Figure~\ref{fig-settled-completion}.
This is natural because the essential property of an integral current space is
its integration and points of $0$ lower density do not contribute to that integration.  In fact integral current spaces are completely settled
with respect to the mass measure, $||T||$, as a consequence
of the requirement that $\set(T)=X$.

The mass of an integral current space, $\mass(T)$, 
is a weighted volume of sorts
which takes into account the integer valued Borel weight defining the current
structure on the space.  When the integral current space is an oriented
Riemannian manifold then its mass is just its volume, $\mass(M)=\vol(M)$.

The boundary of an integral current space is defined
\be
\partial (X,d,T)= (\set(\partial T), d, \partial T)
\ee
where $\partial T$ is the boundary of the integral current defined as in \cite{AK}
so that it satisfies Stoke's Theorem.
When $M$ is a Riemannian manifold, its boundary is just the usual boundary, $\partial M$.

The flat distance between two integral currents is defined as in
\cite{FF}
\be\label{FF}
d_{{F}}(T_1, T_2) = \inf \left\{ \mass(B^{m+1})+\mass(A^m): \, T_1-T_2=A+\partial B\right\}.
\ee
That is $T_1(\omega)-T_2(\omega)= A(\omega)+ \partial B(\omega)=A(\omega)+B(d\omega)$.

\begin{defn}[Sormani-Wenger]
The intrinsic flat distance between integral current spaces is defined in \cite{SorWen2}
as
\be \label{dsw}
d_{\mathcal{F}}\big((X_1, d_1, T_1), (X_2, d_2,T_2)\big)
= \inf\left\{ d^Z_F(\varphi_{1\#}T_1, \varphi_{2\#}T_2) ; \, \varphi_i: X_i \to Z\right\}
\ee
where the infimum is taken over all common complete metric spaces, $Z$,
and all isometric embeddings $\varphi_i: X_i \to Z$ and
where $\varphi_{\#}$ is the push forward map on integral
currents.
\end{defn}

If one constructs a specific $Z$ and isometric embeddings 
$\varphi_i: M_i\to Z$, then one needs only estimate the 
flat distance between the images to obtain an upper bound for
the infimum in (\ref{dsw}).  An explicit filling manifold $B$
satisfying (\ref{Stokes}), then provides an upper bound on the
infimum in (\ref{FF}).  This is how one
obtains the estimate in (\ref{est-int-flat}).  See also
Remark~\ref{rmrk-Z-c}.
 
In \cite{SorWen2} it is proven that this is a distance between precompact
integral current spaces in the sense that
$d_{\mathcal{F}}\big((X_1, d_1, T_1), (X_2, d_2, T_2)\big)=0$ iff there is
a current preserving isometry from $X_1$ to $X_2$.   When the integral
current spaces are oriented manifolds, then there is an orientation
preserving isometry.

Note that all integral current spaces are metric spaces but they need not
be length spaces.  As will be seen
in Example~\ref{ex-not-connected} a sequence of connected
Riemannian manifolds may converge in the intrinsic flat sense to an
integral current space which has broken apart due to the development of
a cusp singularity.  So the limit is not a length metric space.

In \cite{SorWen2} it is proven that if $(M, g_j)$ converge smoothly to
$(M, g_\infty)$ then they converge in the intrinsic flat sense.  In fact, precise
estimates on the intrinsic flat distance are given in terms of the Lipschitz
distance, the diameters and the volumes of the spaces.  The bounds
are found using geometric measure theory.   Here we provide a new
estimate relating the intrinsic flat and Lipschitz distances by
explicitly constructing a filling manifold between them [Lemma~\ref{lem-squeeze-Z}].

If a sequence of oriented Riemannian
manifolds with a uniform upper bound on their volumes
and volumes of their boundaries converges
in the Gromov-Hausdorff sense to a compact metric space $(Y,d)$,
then a subsequence converges
in the intrinsic flat sense to $(X,d,T)$ where $X \subset Y$ and
the metric $d$ is restricted from $Y$
\cite{SorWen2} [Thm 3.20].  In Example~\ref{ex-not-connected}
we see that this may be a proper subset.  In fact the Intrinsic
flat limit may be the $(0,0,0)$ integral current space if $(Y,d)$ has
a lower dimension than the manifolds in the sequence \cite{SorWen2}.

In \cite{SorWen1} two theorems were proven indicating when the
intrinsic flat limit and the Gromov-Hausdorff limits agree.  These theorems
will be reviewed later in the paper as they are applied.   We will also
apply the techniques in their proofs to prove Theorem~\ref{Ricci-codim-thm}.

\subsection{The Scalable Intrinsic Flat Distance}\label{subsect-scalable}

The scalable intrinsic flat distance was suggested as a notion in
work of the second author with Dan Lee \cite{LeeSormani1}
following a recommendation of Lars Andersson.   It is defined to
scale with distance when the Riemannian manifolds are rescaled.
In particular,
\be \label{est-s-int-flat}
d_{s\mathcal{F}}(M^m_1, M^m_2) \le 
\left(\vol_m(A^m)\right)^{1/m}+\left(\vol_{m+1}(B^{m+1})\right)^{1/(m+1)}
\ee
whenever there exist 
isometric embeddings, $\varphi_i: M^m_i \to Z$, into a common complete
metric space, $Z$, and one finds a 
filling submanifold, $B^{m+1}\subset Z,$ and an
excess boundary submanifold, $A^m\subset Z$, satisfying
(\ref{Stokes}).

\begin{rmrk}\label{rmrk-Z-s}
In the setting of Remark~\ref{rmrk-Z-c}
depicted in Figure~\ref{fig-defn-GH} we see that
\begin{eqnarray*}
d_{s\mathcal{F}}(M_1, M_2) &\le& \vol_{m+1}(U \times [0,h],g')^{1/(m+1)}\\
&+& 
\left(\vol_m(M_1\setminus U_1) + \vol_m(M_2\setminus U_2)
+\vol_{m}(\partial U \times [h_1,h_2],g')
\right)^{1/m}.
\end{eqnarray*}
\end{rmrk}

More precisely:

\begin{defn}
The scalable
 intrinsic flat distance between integral current spaces is defined
as
\be
d_{s\mathcal{F}}((X_1, d_1, T_1), (X_2, d_2,T_2))
= \inf\left\{ d^Z_{sF}(\varphi_{1\#}T_1, \varphi_{2\#}T_2) ; \, \varphi_i: X_i \to Z\right\}
\ee
where the infimum is taken over all common complete metric spaces, $Z$,
and all isometric embeddings $\varphi_i: X_i \to Z$ and
where $\varphi_{\#}$ is the push forward map on integral
currents and where the scalable flat distance between $m$
dimensional integral currents is defined by
\be
d_{sF}(T_1,T_2) = \inf \left\{\mass(B)^{1/(m+1)}+  \mass(A)^{1/m} : T_1-T_2=A +\partial B \right\}
\ee
\end{defn}


\section{Examples} \label{Sect-Examples}

\textcolor{blue}{The examples in the section are all fine.}

The following examples are presented to indicate how little control
one may have on limits of manifolds which converge smoothly
away from singular sets and to prove the necessity of the conditions
in our theorems.  The proofs of these examples will sometimes rely on
our theorems proven below but we include them up front so that they
can be kept in mind when reading the remainder of the paper.

\subsection{Losing a Region}

\begin{example}\label{ex-region}
There are metrics, $g_j$, on the sphere, $M^3$, such that
$(M^3, g_j)$ converge smoothly away from $S=\bar{B}_{p_0}(\pi/16)$,
such that the metric completion of the
smooth limit away from $S$ is $S^3\setminus B_{p_0}(\pi/16)$,
the standard round sphere
$(S^3,g_0)$ with a ball removed.   The smooth limit of $M^3$ without the
singular set removed is the entire round sphere and this agrees
with the intrinsic flat and Gromov-Hausdorff
limits (c.f. Lemma~\ref{lem-squeeze-Z}).
\end{example}

\begin{proof}
Taking the metrics, $g_j=g_0$, we have a constant
sequence of standard spheres.   So the intrinsic flat and Gromov-Hausdorff
limits are clearly the standard sphere.   Furthermore
$(g_j, M\setminus S)$ clearly converges to 
$(g_0, M\setminus \bar{B}_{p_0}(\pi/16))$ whose metric completion
is $(g_0, M\setminus {B}_{p_0}(\pi/16))$.
\end{proof}

\subsection{Cones and Cusps}

\begin{example}\label{ex-cone}  
There are metrics $g_j$ on the sphere $M^3$ such that
$(M^3, g_j)$ converge smoothly away from a point singularity
$S=\{p_0\}$ and the
metrics $g_j$ form a conical singularity at $p_0$.
The Gromov-Hausdorff and intrinsic flat limits
agree with the metric completion of $(M\setminus S, g_\infty)$
which is the sphere including the conical tip.
\end{example}

\begin{proof}
More precisely the metrics $g_j$ are defined by
\be
g_j= dr^2 + f_j^2(r) g_{S^2} \textrm{ for } r\in [0,\pi]
\ee
where $f_j(r) = \left( 1/j \right)\sin(r) + \left( 1 - 1/j \right)f(r)$ in which, $f(r)$ is a smooth function such that:
\be
    f(r) = \sin(r)  \; \text{for} \; r \in [0,\pi/2],
\ee
and,
\be
    f(r) = - \frac{2}{\pi} \left( r - \pi \right)  \; \text{for} \; r \in [3\pi/4 , \pi].
\ee
For any $\delta>0$, $f_j$ converge to $f$ smoothly on $[0, \pi-\delta]$.
Thus $g_j$ converge smoothly on compact subsets of $M\setminus S$
to
\be
g_\infty = dr^2 + f^2(r) g_{S^2}.
\ee
The metric completion of $(M\setminus S, g_\infty)$ 
then adds in a single point $p_0$
at $r=\pi$.   Since
\be
\liminf_{r\to 0} \mu(B_{p_0}(r))/r^3 = \frac{4}{3\pi^2} vol (S^2) = \frac{16}{3\pi}>0,
\ee
the point, $p_0$, 
is also included in the settled completion of $(M\setminus S, g_\infty)$.
To complete the proof of the claim we could apply
Theorem~\ref{codim-thm}.
\end{proof}

\begin{example}\label{ex-cusp}
There are metrics $g_j$ on the sphere $M^3$ such that
$(M^3, g_j)$ converge smoothly away from a point singularity
$S=\{p_0\}$ and the
metrics $g_j$ form a cusp singularity at $p_0$.
The Gromov-Hausdorff
agree with the metric completion of $(M\setminus S, g_\infty)$
 which is the
sphere including the cusped tip.
However the intrinsic flat limit
of $(M\setminus S, g_\infty)$ does not include the cusped tip
because it has $0$ density.  So the intrinsic flat
limit is the settled completion of $(M\setminus S, g_\infty)$
which in this case is $(M\setminus S, g_\infty)$
 \end{example}

\begin{proof}
More precisely the metrics $g_j$ are defined by
\be
g_j= dr^2 + f_j^2(r) g_{S^2} \textrm{ for } r\in [0,\pi]
\ee
where $f_j(r) = \left( 1/j \right)\sin(r) + \left( 1 - 1/j \right)f(r)$ in which, $f(r)$ is a smooth function such that:
\be
    f(r) = \sin(r)  \; \text{for} \; r \in [0,\pi/2],
\ee
and,
\be
    f(r) = \frac{4}{\pi^2} \left( r - \pi \right)^2  \; \text{for} \; r \in [3\pi/4 , \pi].
\ee
For any $\delta>0$, $f_j$ converge to $f$ smoothly on $[0, \pi-\delta]$.
Thus $g_j$ converge smoothly on compact subsets of $M\setminus S$
to
\be
g_\infty = dr^2 + f^2(r) g_{S^2}.
\ee
The metric completion of $(M\setminus S, g_\infty)$ 
then adds in a single point $p_0$
at $r=\pi$.   Since
\be
\liminf_{r\to 0} \mu(B_{p_0}(r))/r^3 = \liminf_{r\to 0} \frac{4}{5\pi^2}r^2 vol (S^2) = 0,
\ee
the point, $p_0$, is not included in the settled completion
of $(M\setminus S, g_\infty)$.   

This Gromov-Hausdorff and Intrinsic Flat limits in this example
were proven to be as claimed in the Appendix of \cite{SorWen2}.
One may also apply Theorem~\ref{thm-subdiffeo} to reprove this.
\end{proof}

\subsection{Not Connected}

\begin{example}\label{ex-not-connected}
There are smooth metrics $g_i$ on the sphere, $M^3$,
converging smoothly away from the equator, $S$, such that the
equator pinches to $0$.   Then $(M^3 \setminus S, g_i)$ has
two components, each converging to a standard sphere with
a point removed.  The metric completion of each of the two disjoint
metric spaces is a standard sphere.
However the Gromov-Hausdorff limit is a pair
of spheres joined at a point singularity.  So we see why connectedness
of $M^3 \setminus S$ is a necessary condition in Theorem~\ref{flat-to-settled}.
Here the singular set is of codimension 1.
\end{example}

\begin{rmrk}
In upcoming work of the first author \cite{Lakzian-thesis}, appropriate gluings of
disjoint metric spaces are taken to recover the Gromov-Hausdorff
limit when $M\setminus S$ is not connected.
\end{rmrk}

\begin{proof}
Let $\phi(x)$ be a smooth bump function on $\R$ with the following properties: 
\be
	\int_{ - \infty}^{\infty} \; \phi(x) \; dx = 1
\ee
\be
	\lim_{\epsilon \to 0} \phi_\epsilon(x) = \lim_{\epsilon \to 0} \epsilon^{-1}\phi(x/\epsilon) = \delta_0(x),
\ee
where $\delta_0(x)$ is the Dirac delta function at $0$. 
Let 
\be
\Phi_{1/i} \left( |\sin(2x)| \right) (r) = \phi_{1/i}(x) * | \sin(2x) | (r)= \int_{- \infty}^{\infty} \; \phi_{1/i}(r - x) \; |\sin(2x)| \; \mathrm{d}x.
\ee
It is standard that the sequence is smooth and converges to $| \sin(2r) |$ as $i \to \infty$.  Now, take a partition of unity $\{ \psi , 1 - \psi \}$ on $[0 , \pi]$ such that $\operatorname{supp}(\psi) \subset [\pi/8 , 7\pi/8]$ and $\psi = 1$ on $[\pi/4 , 3\pi/4]$. 
We take the sequence of metrics
\be
g_i= dr^2 + f_i^2(r) g_{S^2}
\ee
where,
\be
	f_i(r) =  \frac{1}{i} \sin(r) +  \frac{i-1}{2i}  \biggl( \left( 1 - \psi(r) \right)\left( |\sin(2r)| \right) + \psi(r) \; \Phi_{1/i} \left( |\sin(2x)| \right) (r) \biggr)
\ee
These are smooth metrics for $r\in [0,\pi]$ because $f_i(r)>0$ for $r\in (0,\pi)$,
\be
f_i'(0)= \frac{1}{i}  + \frac{2(i-1)}{2i} =1,
\ee
\be
    f_i'(\pi)= - \frac{1}{i}  - \frac{2(i-1)}{2i} = - 1
\ee
and
$
f_i''(0)=f_i''(\pi)=0.
$
As $i\to \infty$, $g_i$ converge smoothly away from $r^{-1}(\pi/2)$
to
\be
g_\infty=dr^2 + \frac{\sin^2(2r)}{4} g_{S^2}
\ee
which is a metric on a pair of spheres, each with a point removed.
The metric completion keeps the pair of spheres disjoint, endowing
each with its own point of completion.

The Gromov-Hausdorff and intrinsic flat limits however are a
connected pair of spheres
joined at point which creates a conical singularity.
This can be seen because the distances $d_i$ defined
on $M_i$ using $g_i$ are in fact converging in the Lipschitz
sense to $d_\infty$ defined by using the infimum of lengths, $L_\infty$,
of curves between points where
\be
L_\infty(C) =\int_0^1 g_\infty(C'(s),C'(s))^{1/2} \, ds.
\ee

Taking $W_j=r^{-1}[0, \pi/2 - 1/j] \cup r^{-1}[\pi/2+1/j, \pi]$, then
we have smooth convergence on $W_j$.  The uniform
embeddedness constants converge to $0$.   Both
$\vol_{g_i}(V\setminus W_j)< V_j$ with $V_j \to 0$ and 
$\vol_{g_i}(\partial W_j) \le A_j$ with $A_j \to 0$.   So we
only fail the connectedness hypothesis of this theorem.
\end{proof}

\subsection{Bubbling}

\begin{example}\label{ex-not-F}
There are smooth metric $g_i$ on
the sphere, $M^3$ converging smoothly away
from the singular set $S=\left\{p_0\right\}$ to a sphere.
Yet $(M^3, g_i)$ converge in the Gromov-Hausdorff
and intrinsic flat sense to a pair of spheres
meeting at $p_0$.  See Figure~\ref{fig-not-F}.
\end{example}

\begin{figure}[h] 
   \centering
   \includegraphics[width=5in]{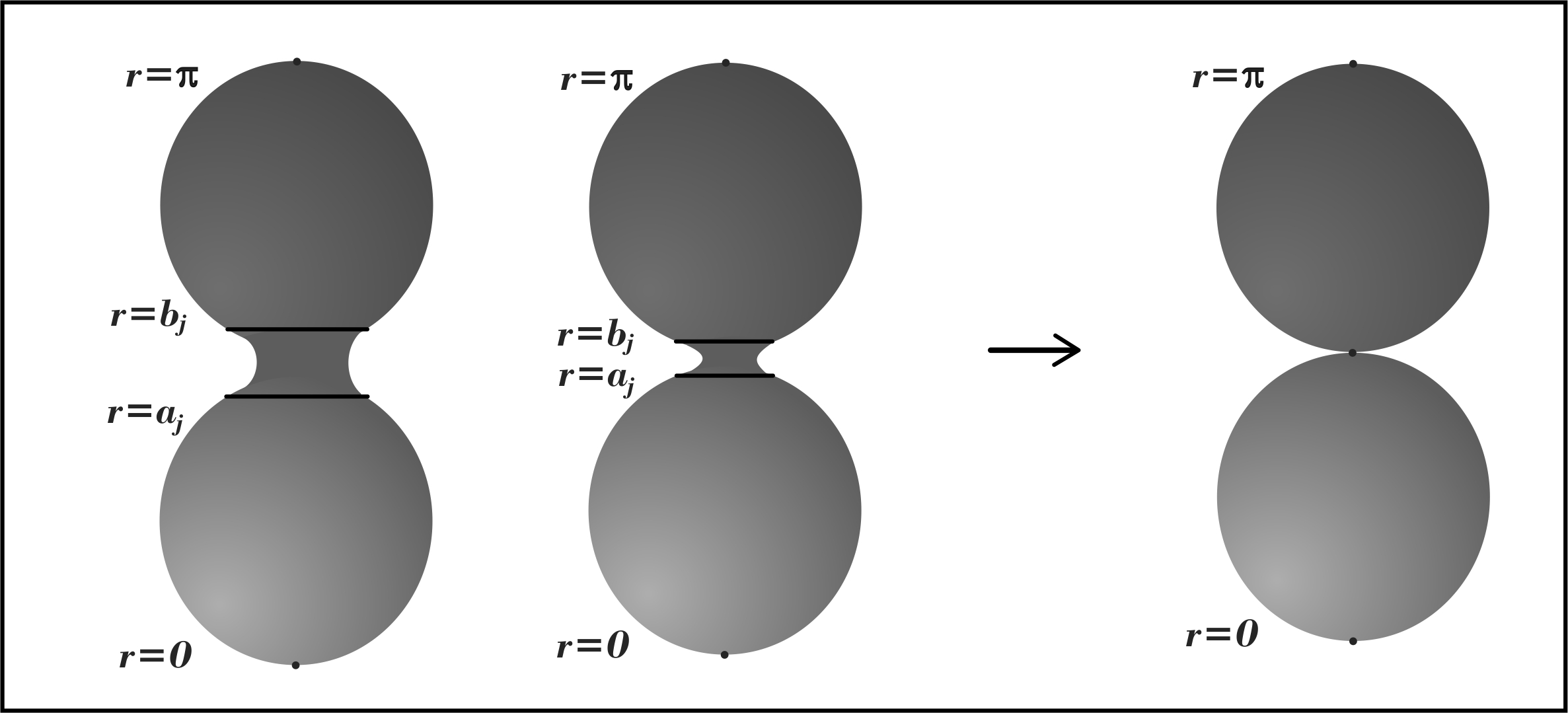}
   \caption{Example~\ref{ex-not-F}.}
   \label{fig-not-F}
\end{figure}

\begin{proof}
Let
\be
g_i= h_i^2(r) dr^2 + f_i^2(r) g_{S^2}
\ee
where
\begin{eqnarray}
h_i(r)=1 &\textrm{ on } & r\in [0,a_i]\\
f_i(r)=\sin(r) &\textrm{ on } & r\in [0,a_i]
\end{eqnarray}
where $a_i= \pi-  \pi/(10i)$
so that $g_i$ converges smoothly away from $S$ to the
round metric $g_\infty$ on the sphere.  The metric completion
of $(M\setminus S, g_\infty)$ is the round sphere.

Now we set
\begin{eqnarray}
h_i(r)=10i &\textrm{ on } & r\in [b_i, \pi]\\
f_i(r)=\sin((\pi -r ) /(10i)) &\textrm{ on } & r\in [b_i,\pi]
\end{eqnarray}
where $b_i=\pi- ( \pi - \pi / (10i) ) / (10i)$ so there is symmetry
and we extend them smoothly for $r\in [a_i, b_i]$
so that
\be
h_i(a_i) =1 \le h_i(r)\le 10i=h_i(b_i)
\ee
and
\be
0 < f_i(r) < \max \{f_i(a_i), f_i(b_i)\}.
\ee
These thin regions are converging to
a single point.  So the Gromov-Hausdorff limit of $(M, g_i)$
is a pair of standard spheres joined at a point and the
intrinsic flat limit is the same.  The smooth limit away from
$S$ missed the second sphere!
\end{proof}

\subsection{Losing Volume in the Limit}

\begin{example} \label{ex-to-hemisphere}
There are $(M^3, g_i)$ all isometric to the standard sphere which
converge smoothly away from a singular set $S=\{p_0\}$ to
$(M\setminus S, g_\infty)$ which is isometric to an open hemisphere.
The metric completion agrees with the settled completion, $(M_\infty, d_\infty)$
which is isometric to a closed hemisphere.  The singular set is
codimension $2$ in $M$.   This example satisfied all the
conditions of all of our Theorems concerning smooth
convergence away from singular sets except
$\vol(M\setminus W_j)< V_j$ where $\lim_{j\to\infty} V_j=0$.
\end{example}

\begin{proof}
Again we view $M^3=S^3$ as a warped product with a warping function
$r\in [0, \pi]$, such that $r(p_0)=\pi$.   Let
\be
g_i = (h_i'(r) )^2 dr^2 + \sin^2(h_i(r)) g_{S^2}
\ee
where $h_i(r)$ is a smooth increasing function such that
\begin{eqnarray}
h_i(r) &=& r (\pi/2)/(\pi-\pi/(2i)) \textrm{ for } r\in[0, \pi - 1/i] \\
h_i(r) &=& \pi - (r-\pi)(1/(2i))(\pi-1/(2i)) \textrm{ for } r\in [\pi - 1/(2i) , \pi].
\end{eqnarray}
Then the diffeomorphism which maps $r \mapsto s=h_i(r)$ is an isometry
from $(M^3, g_i)$ to $(S^3, g_{S^3})$.

On any compact set $K \subset M\setminus S$, there exists a $j$
sufficiently large that $K \subset r^{-1}[0, \pi-1/j]$.  Taking $i \to \infty$
we see that on $K$, $h_i(r) \to r/2$ and $g_i$ converge smoothly to
\be
g_\infty= (1/2)^2 dr^2 + \sin^2(r/2) g_{S^2}.
\ee
Thus $(M\setminus S, g_\infty)$ is isometric to an open hemisphere
via the isometry which maps $r \mapsto s=r/2$.   The metric
completion is then the closed hemisphere and the settled completion
agrees with the metric completion because every point in the closed
hemisphere has positive lower density.

Setting $W_j =r^{-1}[0, \pi-1/j]$, we see that 
\be
\vol_{g_i}(\partial W_j) \le 4\pi.
\ee
Clearly the diameter, volume, Ricci curvature and contractibility
conditions all hold because the sequence of $(M, g_j)$ are all
isometric to spheres.   However
\be
\lim_{i\to \infty} \vol_{g_i}(M\setminus W_j) \ge \vol(S^3)/2.
\ee
\end{proof}

\subsection{Unbounded Volumes and Diameters}

Recall that below Theorem~\ref{Ricci-codim-thm}, we stated that
the diameter condition is not necessary when the manifold has nonnegative
Ricci curvature.   Here we see that the volume bound is still necessary:

\begin{example}\label{ex-cap-cyl}
There are metrics $g_i$ on the sphere $M^3$  with nonnegative Ricci curvature such that
$(M^3, g_i)$ converge smoothly away from a point singularity
$S=\{p_0\}$ to a complete noncompact manifold; In particular, 
converging to a hemisphere attached to a cylinder of length $k$ on the $r^{-1}[0, \pi-1/k)$ region.   
\end{example}

\begin{proof} 
For any $L \in \R$ large enough, define the warped metric $g_{L}$ on $[0,L] \times S^2$ as follows:
\be
	g_L(t) = dt^2 + \left( f_L(t) \right)^2 g_{S^2}
\ee
where,
\be
	f_L(t) = \sin(t)  \;\; \text{for}\;\; t \in [0 , \pi/2 ]
\ee
\be
	f_L(t) = 1  \;\; \text{for}\;\; t \in [\pi/2+1/100 , L - \pi/2-1/100 ]
\ee
\be
	f_L(t) = \sin(\pi + t - L)  \;\; \text{for}\;\; t \in [L - \pi/2 , L ]
\ee
and $f_L(t)$ smooth with $f_L''(t)<0$ elsewhere.
We will be calling $g_L$, the double torpedo metric (it is comprised of two torpedo metrics glued together from their cylindrical ends.) For any $L$, 
$g_L$ has nonnegative Ricci curvature. 

Let $\phi : [0,\pi] \to [0,\infty)$ be a smooth increasing function such that
\be \label{phi-1}
	\phi(r) = r    \;\; \text{for} \;\; r \in [0 , \pi/2]
\ee
with 
\be\label{phi-2}
	\lim_{r \to \pi} \phi(r) = \infty
\ee

For $ j > 2$, let $\phi_j(r): [0 , \pi] \to [0 , L_j = j + \pi/2 + 1]$ be a smooth increasing function such that
\be \label{phi-3}
	\phi_j(r) = \phi(r)  \;\; \text{for} \;\; r \in [0 , \phi^{-1}(j + \pi/2)],
\ee
and
\be \label{phi-4}
	\phi_j (r) = j + r-\pi/2 + 1 \textrm{ for } r \textrm{ near } \pi.
\ee

Again, we view $M^3 = S^3$ as a warped product with a warping function $r\in [0, \pi]$, such that $r(p_0)=\pi$.  Let 
\be
g_j(r) = \phi_j^{*} \left( g_{L_j} \right) = (\phi_j'(r) )^2 dr^2 + \left( f_{L_j} \right)^2(\phi_j(r)) g_{S^2}
\ee
Then the diffeomorphism $\phi_j$ is an isometry from $(M^3, g_j)$ to 
$(S^3, g_{L_j})$. On any compact set $K \subset M\setminus S$, there exists a $k$ sufficiently large that $K \subset r^{-1}[0, \pi-1/k]$.  Taking $j \to \infty$ we see that on $K$, $g_j$ converge to 
\be
g_\infty= (\phi'(r) )^2 dr^2 +  f^2(\phi(r)) g_{S^2}.
\ee
where,
\be
	f(r) = \sin(r) \;\; \text{for}\;\; r \in [0 , \pi/2-1/100]  \;\; \text{and}\;\; f(r) = 1 \;\; \text{for}\;\; r \in [\pi/2 +1/100, \infty)
\ee
which is a hemisphere smoothly attached to a cylinder of length $k$.

If we take $W_j = r^{-1}([0 , \pi - 1/j))$ then, we see that, $\vol_{g_i}(W_j)$ and $\vol_{g_i}(M \setminus W_j)$ are unbounded.  Since $(M \setminus S , g_\infty)$ is complete, it coincides with the metric completion. Since $(M \setminus S , g_\infty)$ is noncompact , $(M^3 , g_j)$ does not have Gromov - Hausdorff limit. Also since the volume is not finite, there is no intrinsic flat limit either.

Nevertheless, this example has $\vol(\partial W_j)\le 4\pi$ and $W_j$
are uniformly embedded, the sequence has nonnegative Ricci curvature
and a uniform contractibility function, $\rho(r)=r$ for $r\in (0, \pi/2]$.    
\end{proof}

\begin{example}\label{ex-diam-now}
There are metrics $g_i$ on the sphere $M^3$  with $Ricci\ge (n-1)H g$ 
such that
$(M^3, g_i)$ converge smoothly away from a point singularity
$S=\{p_0\}$ to a complete noncompact manifold; In particular, 
converging to a hemisphere attached to an infinitely long cusp.
\end{example}

\begin{proof} 
Let $h:[0, \infty) \to [0,\infty)$ be defined so that
$h(t)=\sin(t)$ for $t\in [0,\pi/2]$ and
$h(t) =e^{-t}$ for $t\in [\pi, \infty)$
and smooth in between so that 
\be
g= dt^2 + h^2(t) g_{S^2}
\ee
is a complete noncompact metric with finite volume over 
$[0, \infty)\times S^2$.   Observe that the sectional curvature
is uniformly bounded below by some negative constant, $H$.

For any $L \in \R$ large enough, we can find $\epsilon_L>0$
sufficiently small so that we may 
define a smooth 
warped metric $g_{L}$ on $[0,L] \times S^2$ as follows:
\be
	g_L(t) = dt^2 + \left( f_L(t) \right)^2 g_{S^2}
\ee
where,
\be
	f_L(t) = h(t)  \;\; \text{for}\;\; t \in [0 ,  L - 2\epsilon_L ]
\ee
and 
\be
	f_L(t) = \sin(\pi + t - L)  \;\; \text{for}\;\; t \in [L -\epsilon_L , L ]
\ee
and $f_L(t)$ smooth with $-f_L''(t)/f_L(t)>H$ elsewhere.
For any $L$, $g_L$ has sectional curvatures $\ge H$. 

Let $\phi : [0,\pi] \to [0,\infty)$ be a smooth increasing function 
as in the prior example.  In particular 
satisfying (\ref{phi-1}), (\ref{phi-2}), (\ref{phi-3})
and (\ref{phi-4}).

Again, we view $M^3 = S^3$ as a warped product with a warping function $r\in [0, \pi]$, such that $r(p_0)=\pi$.  Let 
\be
g_j(r) = \phi_j^{*} \left( g_{L_j} \right) = (\phi_j'(r) )^2 dr^2 + \left( f_{L_j} \right)^2(\phi_j(r)) g_{S^2}
\ee
Then the diffeomorphism $\phi_j$ is an isometry from $(M^3, g_j)$ to 
$(S^3, g_{L_j})$. On any compact set $K \subset M\setminus S$, there exists a $k$ sufficiently large that $K \subset r^{-1}[0, \pi-1/k]$.  Taking $j \to \infty$ we see that on $K$, $g_j$ converge to 
\be
g_\infty= (\phi'(r) )^2 dr^2 +  h^2(\phi(r)) g_{S^2}.
\ee
Since $(M \setminus S , g_\infty)$ is complete, it coincides with the metric completion. Since $(M \setminus S , g_\infty)$ is noncompact , $(M^3 , g_j)$ does not have Gromov - Hausdorff limit. 

If we take $W_j = r^{-1}([0 , \pi - 1/j))$ then, we see that, $\vol_{g_i}(W_j)$ and $\vol_{g_i}(M \setminus W_j)$ are bounded, $\vol(\partial W_j)\le 4\pi$ and 
$W_j$ are uniformly embedded.  So this proves the necessity of the 
diameter condition in Theorem~\ref{Ricci-codim-thm}.
\end{proof}

\subsection{Spheres with Splines}

The following examples are based upon examples in \cite{SorWen2}.

\begin{example}\label{ex-not-GH}
There are metrics $g_i$ on the sphere $M^3$ such that
$(M^3, g_i)$ converge smoothly away from a point singularity
$S=\left\{p_0\right\}$ yet we have
a single spline of finite length, $L$, becoming thinner and
thinner
so that the Gromov-Hausdorff limit is not the sphere while the 
intrinsic flat limit is just the sphere.   The metric completion of
$(M\setminus S, g_\infty)$ is also the sphere in this example.
\end{example}

A version of this example with positive scalar curvature will be given
in \cite{Lakzian-Diameter}.

\begin{proof}
More precisely the metrics $g_i$ are defined by
\be
g_i= h_i(r)^2dr^2 + f_i^2(r) g_{S^2} \textrm{ for } r\in [0,\pi]
\ee
where $f_i(r)= \sin(r)$ and
\be
h_i(r)= 1 + i \, \exp \left( \tfrac{ \left( \tfrac{1}{2i} \right)^2 }
{\left(r - \pi + \tfrac{2}{i}\right)\left(r - \pi + \tfrac{1}{i}\right)} \right) \,\,\chi_{\left[ \pi - \tfrac{2}{i} , \pi - \tfrac{1}{i}) \right]}.
\ee
Observe that on $r^{-1}[0, \pi-1/j)$ we have $h_i(r) =1$ for $i\ge 2j$.
So $g_i$ converge smoothly away from $p_0$ to the standard metric
on a sphere, $g_\infty$.  The metric  and settled completions of $(M\setminus \{p_0\}, g_\infty)$ are both the standard sphere.   

We will refer to $N_i=r^{-1}(\pi-2/i, \pi]$ with the metric $g_i$ as a {\em spline}.
Observe that
\be
\diam_{g_i}(M_j) \ge \int_0^\pi h_i(r) \, dr =
\pi-2/i + L +1/i 
\ee
where $L$ is the length of the spline:
\begin{eqnarray}
L&=&\int_{\pi-2/i}^{\pi-1/i} h_i(r) \, dr \\
&=&  \int_0^1 1+ e^{1/(4u(u-1))} \, du
\end{eqnarray}
Since the diameter of the Gromov-Hausdorff limit, when it exists,
is the limit of the diameters of the sequence, we see that the
Gromov-Hausdorff limit is not metric completion in this case.  We
will not provide an explicit proof that the Gromov-Hausdorff limit
is in fact the sphere with a line segment of length $D$ attached at
$p_0$.

Now taking $W_j = r^{-1} ([0 , \pi - 1/(2j)])$, we see that 
\be
\diam_{M_i} (W_j) \le \pi \textrm{ for }i \ge j,
\ee
\be
\vol(M_i) \le \vol(S^3, g_0) + \sin(1/(2i))L  \le V_0
\ee
\be
\vol_{g_i}(\partial W_j) \le 4\pi
\ee
\be
\vol_{g_i}(M\setminus W_j) \le \pi (1/(2j)^2) + \pi \sin(2/j)^2 L \le V_j 
\ee
where $\lim_{j \to \infty} V_j = 0$.   By Theorem~\ref{codim-thm}
we have the intrinsic flat limit is settled metric completion 
which is the sphere.   This example has no uniform
lower bound on Ricci curvature nor a uniform contractibility function
so it demonstrates the necessities of these conditions in all
of our theorems which require them to prove the Gromov-Hausdorff
limit exists and is the metric completion of $(M\setminus S, g_\infty)$.
\end{proof}

\begin{example}\label{ex-no-GH}
There are metrics $g_i$ on the sphere $M^3$ 
with uniformly bounded diameter and volume such that
$(M^3, g_i)$ converge smoothly away from a point singularity
$S=\left\{p_0\right\}$ and we have
increasingly many splines of length $L$ whose total volume goes to $0$
based in smaller and smaller neighborhoods of $S$.
The metric
completion of $(M\setminus S, g_\infty)$ is the round
sphere.  This is also the intrinsic flat limit.
The Gromov-Hausdorff limit, however, does not exist
since the number of balls of radius $L/2$ diverges to infinity.
\end{example}

A version of this example with positive scalar curvature will be given
in \cite{Lakzian-Diameter}.

\begin{proof}
Let $(M,g_i)$ be created by taking the standard sphere of radius 
$1$ and removing $i$ pairwise disjoint balls of radius $2/i^2$ from the ball of
radius $2/i$ about $p_0$.   Replace each of those balls with
a spline, $N_{i^2}$ from the previous example.  Each spline has
length $L$ as in the previous example, so there are $i$ balls of radius
$L/2$ centered at the tips of the splines.  By Gromov's Compactness
Theorem's Converse, there is no subsequence converging in the
Gromov-Hausdorff sense.

However, $\diam(M, g_i)\le \pi + 2L$.

Each $(M, g_i)$ is diffeomorphic to $S^3$, via the identity map
outside of the splines and via the diffeomorphism from each spline to
the ball it has replaced.   Taking any precompact set $W\subset S^3$, such that
$p_0\notin \bar{W}$, we can take $i$ sufficiently large that
$W\cap B_{p_0}(2/i)=\emptyset$, so that $g_i$ is then the standard
metric on $W_i$.  So we see that $(M, g_i)$ converges smoothly
to a standard sphere with $p_0$ removed.  The metric and
settled completions are one again the standard sphere.

Let $W_j= S^3 \setminus B_{p_0}(2/j)$ where the ball is measured using
the standard metric on $S^3$ so that for $j \le i$ there are no splines
within $(W_j, g_i)$.
\be
\diam_{M_i} (W_j) \le \pi \textrm{ for }i \ge j,
\ee
\be
\vol(M_i) \le \vol(S^3, g_0) + i \sin(1/(2i^2))L  \le V_0
\ee
\be
\vol_{g_i}(\partial W_j) \le 4\pi \sin(1/2i^2) i \le A_0
\ee
\be
\vol_{g_i}(M\setminus W_j) \le \pi i (1/(2i^2)^2) + \pi i \sin(2/i^2)^2 L \le V_j 
\ee
where $\lim_{j \to \infty} V_j = 0$.    By Theorem~\ref{codim-thm}
we have the intrinsic flat limit is settled metric completion 
which is the sphere.   This example has no uniform
lower bound on Ricci curvature nor a uniform contractibility function
so it demonstrates the necessities of these conditions in all
of our theorems which require them to prove the Gromov-Hausdorff
limit exists and is the metric completion of $(M\setminus S, g_\infty)$.
\end{proof}

\begin{example}\label{ex-not-bounded}
There are metrics $g_i$ on the sphere $M^3$
with uniformly bounded volume such that
$(M^3, g_i)$ converge smoothly away from a point singularity
$S=\left\{p_0\right\}$ and we have
a single spline of increasing length whose volume goes to $0$
and width goes to $0$
contained in smaller and smaller neighborhoods of $S$.
The metric
completion of $(M\setminus S, g_\infty)$ is the round
sphere.  This is also the intrinsic flat limit.
The Gromov-Hausdorff limit, however, does not exist
since the diameter diverges to infinity.
\end{example}

\begin{proof}
More precisely the metrics $g_i$ are defined by
\be
g_i= h_i(r)^2dr^2 + f_i^2(r) g_{S^2} \textrm{ for } r\in [0,\pi]
\ee
where $f_i(r)= \sin(r)$ and
\be
h_i(r)= 1 + i^2 \, \exp \left( \tfrac{ \left( \tfrac{1}{2i} \right)^2 }
{\left(r - \pi + \tfrac{2}{i}\right)\left(r - \pi + \tfrac{1}{i}\right)} \right) \,\,\chi_{\left[ \pi - \tfrac{2}{i} , \pi - \tfrac{1}{i}) \right]}.
\ee
Observe that on $r^{-1}[0, \pi-1/j)$ we have $h_i(r) =1$ for $i\ge 2j$.
So $g_i$ converge smoothly away from $p_0$ to the standard metric
on a sphere, $g_\infty$.  The metric  and settled completions of $(M\setminus \{p_0\}, g_\infty)$ are both the standard sphere.   

Observe that
\be
\diam_{g_i}(M_j) \ge \int_0^\pi h_i(r) \, dr =
\pi-2/i + L_i +1/i 
\ee
where $L_i$ is the length of the spline:
\begin{eqnarray}
L_i&=&\int_{\pi-2/i}^{\pi-1/i} h_i(r) \, dr \\
&=& i \int_0^1 e^{1/(4u(u-1))} \, du=iL
\end{eqnarray}

Now taking $U_j = r^{-1} ([0 , \pi - 1/(2j)])$, we see that 
\be
\diam_{M_i} (U_j) \le \pi \textrm{ for }i \ge j,
\ee
\be
\vol(M_i) \le \vol(S^3, g_0) + \sin(1/(2i))L_i  \le V_0
\ee
\be
\vol_{g_i}(\partial W_j) \le 4\pi 1
\ee
\be
\vol_{g_i}(M\setminus W_j) \le \pi (1/(2j)^2) + \pi \sin(2/j)^2 L_j \le V_j 
\ee
where $\lim_{j \to \infty} V_j = 0+\lim_{j\to\infty}\sin^(2/j)jL=0$.   
 By Theorem~\ref{codim-thm}
we have the intrinsic flat limit is settled metric completion 
which is the sphere.   This example has no uniform
lower bound on Ricci curvature nor a uniform contractibility function
so it demonstrates the necessities of these conditions in all
of our theorems which require them to prove the Gromov-Hausdorff
limit exists and is the metric completion of $(M\setminus S, g_\infty)$.
\end{proof}

\subsection{Unbounded Boundary Volumes}

Here we have examples demonstrating the necessity of the
$\vol(\partial W)$ conditions in our theorems.

\begin{example} \label{ex-flamenco}
There are $(M^3, g_j)$ all diffeomorphic to the standard sphere which
converge smoothly away from a singular set $S=\{p_0\}$ to
$(M\setminus S, g_\infty)$ with the metric
\be
    dr^2+ f^2(r) g_{S^2}  \textrm{ where } r\in [0,\pi)
\ee
such that $f(r)=\sin(r)$ on $[0,\pi/2]$ and
$\vol_{g_\infty}(M\setminus S)$ is finite but
\be
\lim_{r\to \pi} f(r)=\infty.
\ee
The metric completion agrees with the settled completion of
$(M_\infty, d_\infty)$, which is not an integral current space
because the area of the boundary is infinite. 
The diameter of this example is clearly
$\le 2\pi$.
This example demonstrates that (\ref{area-2}) of Theorem~\ref{codim-thm} 
is a necessary condition.
\end{example}

\begin{proof}
Let $g_{\infty} = dr^2 + f^2(r) g_{S^2}$ where $f(r)$ is a smooth 
increasing function such that:
\be
    f(r) = \sin(r) \; \text{for} \; r \in [0,\pi/2]
\ee
and 
\be
    f(r) = (\pi - r)^{- \frac{1}{4}} \; \text{for} \; r \in [\pi/2 +1/2 , \pi).
\ee
Then we have:
\be
    \lim_{r \to \pi}f(r) = \infty
\ee
and
\be
\vol_{g_{\infty}}(M \setminus S) = \int_0^{\pi} \; \omega_2 f^{2}(r) \; \mathrm{d}r < \infty.
\ee

Now let $g_i = dr^2  + f_i^2(r)g_{S^2}$  where, $f_i(r)\le f(r)$ is a smooth function given by:
\be
   f_i(r) = f(r)  \; \text{for} \; r \in [0,\pi - 1/i] 
\ee
and
\be
   f_i(r) = \sin(r)  \; \text{for} \; r \in [\pi - 1/(2i) , \pi].
\ee
It is easy to see that $g_i$ converges to $g_{\infty}$ away from the singular point. 

Taking $W_j = r^{-1}([0 , \pi - 1/j])$, we see that all conditions 
of Theorem~\ref{codim-thm}
are satisfied except that $\vol_{g_i}(\partial W_j)$ is not bounded.
\end{proof}

\begin{rmrk}\label{maybe-contr-area}
The sequence in Example~\ref{ex-flamenco} 
also appears to satisfy uniform local contractibility
estimates as there is no cusp effect.  
The Gromov-Hausdorff limit appears to be
the one point completion of $(M\setminus S, g_\infty)$.  The
metric completion of $(M\setminus S, g_\infty)$ includes infinitely
many new points.
So this example may
well also prove necessity of boundary volume
estimates in Theorem~\ref{c-codim-thm}.   
\end{rmrk}

\begin{rmrk} \label{no-ex-Ricci-area}
It is an open question whether the area hypothesis, (\ref{area-2}),
is a necessary hypothesis in Theorem~\ref{Ricci-codim-thm}.  It
is possibly that one might always find a new exhaustion satisfying
this condition as long as one has an exhaustion satisfying all the
other hypothesis of the theorem.   
\end{rmrk}

\subsection{Torus to Square}

\begin{example} \label{ex-to-torus-square}
There are $(M^2, g_j)$ all isometric to
the flat torus, $S^1\times S^1$
which converge smoothly away from a singular set
\be
S= \left(S^1\times\{0\} \right) \cup \left( \{0\} \times S^1 \right) \subset S^1\times S^1
\ee
to
\be
(M\setminus S, g_\infty)= \left( (0,2\pi)\times(0,2\pi), dt^2+ds^2\right).
\ee
So the metric completion and the settled completions are both
\be
M_\infty= [0,2\pi]\times[0,2\pi]
\ee
with the standard flat metric, while the intrinsic flat and Gromov
Hausdorff limits are the flat torus $S^1\times S^1$.
Thus the codimension condition and the uniform embeddedness
conditions are necessary in all our theorems.
\end{example}

\begin{proof}
Let $W_k = (1/k, 2\pi - 1/k) \times (1/k, 2\pi - 1/k)$. Then, for $j$ large enough:
\be\label{lambda-ijk}
\lambda_{i,j,k}= \sup_{x,y\in W_j} |d_{W_k}(x,y)- d_{M}(x,y)| = 2\pi - 4/j
\ee
Therefore,
\be
\limsup_{j\to \infty} \limsup_{k\to \infty} \limsup_{i\to\infty} \lambda_{i,j,k} = 2\pi
\ee
and
\be
\limsup_{j\to \infty} \limsup_{i\to \infty} \limsup_{k\to\infty} \lambda_{i,j,k} = 2\pi
\ee
So we fail uniform embeddedness as well as the codimension 2
condition.  

Observe that the sequence satisfies Ricci curvature, contractibility,
diameter and volume conditions on $M_i$ because all the $M_i$
are the standard flat torus.  Furthermore $\vol_{g_i}(M\setminus W_j)\le 4/j$
and $\vol_{g_i}(\partial W_j) \le 4$.
\end{proof}

\section{Explicit Estimates with Isometric Embeddings} \label{Sect-estimates}

\textcolor{blue}{The work in this section is fine.}

In this section we construct isometric embeddings of
Riemannian manifolds into metric spaces to provide explicit
bounds on the Gromov-Hausdorff and intrinsic flat distances
between them.  

Recall the definition of isometric embedding given
in Subsection~\ref{subsect-iso}.  In fact we construct more general mappings.   

\begin{defn}
Let $D>0$ and $M, M'$ are geodesic metric spaces.  
We say that $\varphi: M \to M'$ is a 
\textcolor{black}{$D$-geodesic
embedding}
if for any smooth minimal geodesic,
$\gamma:[0,1]\to M$, of
length $\le D$ we have
\be\label{star-0}
d_{M'}(\varphi(\gamma(0)), \varphi(\gamma(1)))=L(\gamma).
\ee
\end{defn}

When $D=\diam(M)$, then $D$-geodesic embeddings are
isometric embeddings .   The
advantage of this more general notion is that it can be applied
when $M$ is not complete.   This will be essential to proving
Theorem~\ref{thm-subdiffeo}.

\subsection{Hemispherical Embeddings}

In this subsection we prove the following key proposition:

\begin{prop} \label{prop-hem}
Given a manifold $M$ with Riemannian metrics
$g_1$ and $g_2$ and $D_1, D_2, t_1, t_2>0$.  Let
$M'=M\times[t_1,t_2]$ and let $\varphi_i: M_i\to M'$
be defined by $\varphi_i(p)=(p,t_i)$.   If a metric
$g'$ on $M'$ satisfies
\be \label{ineq-prop-hem}
g' \ge  dt^2 +  \cos^2((t-t_i)\pi/D_i) g_i  \textrm{ for } |t-t_i|<D_i/2
\ee
and
\be \label{eq-prop-hem}
g'=dt^2 +g_i \textrm{ on }M\times \left\{t_i\right\}\subset M'
\ee
then any geodesic, $\gamma:[0,1]\to M_i$,
of length $\le D_i$ satisfies (\ref{star-0}).   If, the diameter is bounded,
$\diam_{g_i}(M)\le D_i$, then $\varphi_i$ is an isometric
embedding.

Furthermore, for $q_1, q_2\in M$, we have
\be\label{extra-edges}
d_{M'}(\varphi_1(q_1), \varphi_2(q_2))\ge
d_{M_i}(q_1,q_2).
\ee
\end{prop}

\begin{figure}[h] 
    \centering
    \includegraphics[width=5in]{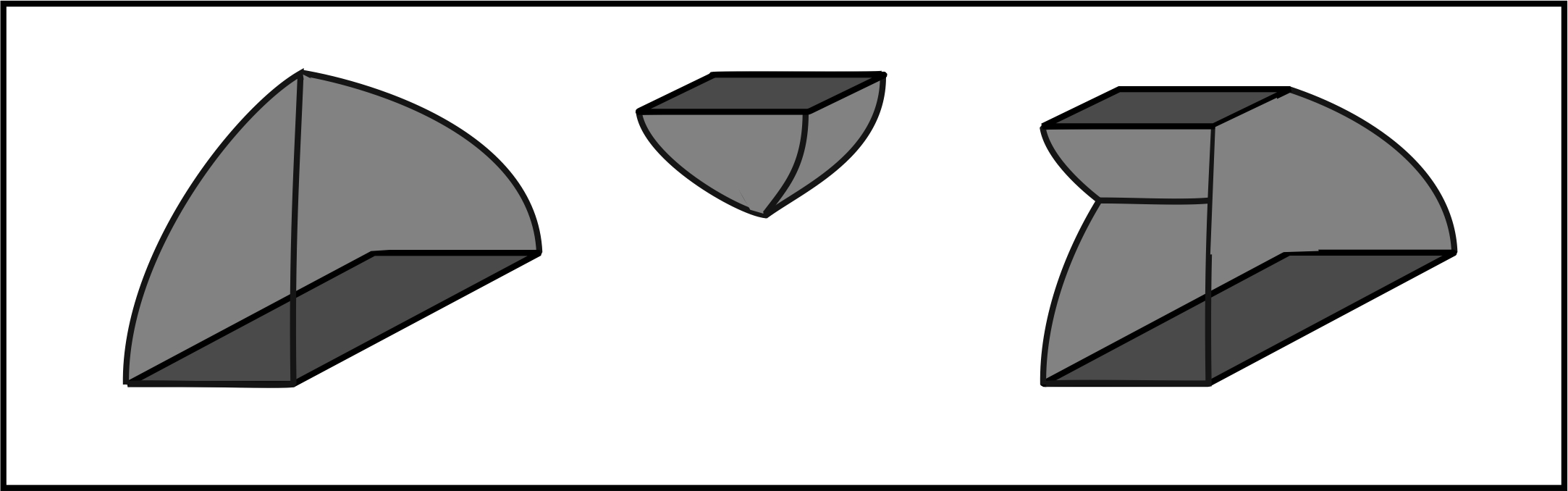}
    \caption{First we see a pair of flat tori, $M_i=(M, g_i)$, isometrically embedded
in their own hemispherical suspensions.
Then they both isometrically embed into a common $M'$.}
    \label{fig-cor-hem-ex}
  \end{figure}

\begin{example}
Let $M=S^1 \times[0,1]$ and let $g_i= a_i^2 d\theta^2 + b_i^2 dl^2$,
where $a_1>a_2>0$ and $b_2>b_1>0$.   Take
\be
D_i = \diam_{g_i}(M_i)=\sqrt{(\pi a_i)^2 + b_i^2}.
\ee
By Proposition~\ref{prop-hem}, we know that if 
we can find $t_i \in \mathbb{R}$ and functions $a,b: [t_1, t_2]\to \mathbb{R}$
such that
\begin{eqnarray}
a(t)&\ge&  \max_{i=1,2} a_i h_i(t)\\
b(t)&\ge&  \max_{i=1,2} b_i h_i(t)
\end{eqnarray}
where $h_i(t)= \max\{\cos((t-t_i)\pi/D_i), 0\}$
and
\be \label{need-this-1}
a(t_1)=a_1, \,\,a(t_2)=a_2, \,\, b(t_1)=b_1 \,\, \textrm{ and } b(t_2)=b_2,
\ee
then we have isometric embeddings $\varphi_i: (M,g_i) \to (M',g')$
where
\be \label{this-this}
g' = dt^2 +a^2(t) d\theta^2 + b^2(t) dl^2.
\ee
To obtain (\ref{need-this-1}), we must choose $t_2-t_1$
sufficiently large that
\be
a_1 h_1(t_2-t_1) \le a_2 \textrm{ and }
b_2 h_2(t_2-t_1) \le b_1.
\ee
Since $a_2/a_1, b_2/b_1 < 1$ this is achieved by taking
\be
|t_2-t_1| \ge \max \left\{ \frac{D_1}{\pi} \arccos\left(\frac{a_2}{a_1}\right),
\frac{D_2}{\pi} \arccos\left(\frac{b_1}{b_2}\right) \right\}.
\ee
See Figure~\ref{fig-cor-hem-ex}.
\end{example}

Before we prove the proposition we prove the
following lemma.   Recall that equators in spheres
isometrically embed into the hemispheres.  Here
we create standard
isometric embeddings of Riemannian
manifolds into hemispherically warped product spaces. 
The idea comes from Gromov's notions in filling
Riemannian manifolds \cite{Gromov-filling}.

\begin{lem} \label{lem-hem}
Given a compact Riemannian manifold $(M,g)$ and $D>0$.  Let
$M'=M\times[0,D/2]$ and let $\varphi: M\to M'$
be defined by $\varphi(p)=(p,0)$.   If a metric
$g'$ on $M'$ satisfies
\be \label{ineq-lem-hem}
g' \ge  dt^2 +  \cos^2(t\pi/D) g  \textrm{ on } M' 
\ee
and
\be \label{eq-lem-hem}
g'=dt^2 +g \textrm{ on }M\times \left\{0\right\}\subset M'
\ee
then any geodesic, $\gamma:[0,1]\to M$,
of length $\le D$ satisfies (\ref{star-0}).   If
$\diam(M)\le D$ then $\varphi$ is an isometric
embedding.
\end{lem}

Here the hemispherical suspension, $M'$, is a well defined
metric space but not necessarily a smooth manifold as can be seen, for example,
on the left side of Figure~\ref{fig-lem-hem}.  The inspiration for using
a hemispherical suspension comes from Gromov's work on filling Riemannian
manifolds \cite{Gromov-filling}.

\begin{proof}
Assume not.  There exists a geodesic $\gamma:[0,1]\to M$
of length $\le D$, and a curve $\sigma: [0,1]\to M'$
running from $\gamma(0)$ to $\gamma(1)$ of length 
$L_{g'}(\sigma)< L_g(\gamma)$.    If we replace the metric
$g'$ by $g''=dt^2 +  \cos^2(t\pi/D) g$, then 
$L_{g''}(\sigma)<L_g(\gamma)$.

So there exists a curve $C(s)=(x(s), t(s)) \in M\times [0,D/2]$
which is minimizing with respect to $g''$
between its endpoints $C(0)=(x(0),0)=\gamma(0)$
and $C(1)=(x(1), 0)=\gamma(1)$ such that 
\be\label{no-no}
L_{g''}(C) \le L_{g''}(\sigma)\le L_{g'}(\sigma)< L_g(\gamma)=d_M(x(0), x(1))\le D.  
\ee
Since $C:[0,1]\to M'$
is a minimizing geodesic in the warped product, $x:[0,1]\to M$, is a
minimizing geodesic in $M$.  We choose the parameter $s$ so that $x$
is parametrized proportional to arclength and let $h$ be the length
of the geodesic $x$, so we have $h=d_M(x(0), x(1)) \le D$. See Figure~\ref{fig-lem-hem}.

\begin{figure}[h] 
   \centering
   \includegraphics[width=5in]{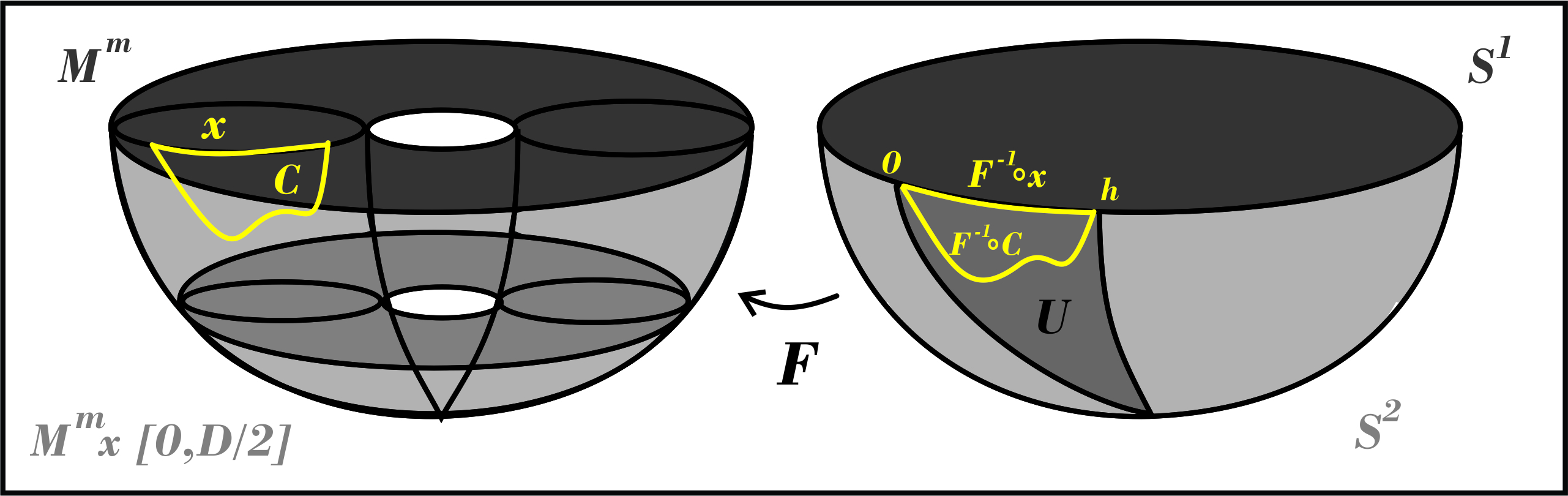}
   \caption{On the left we have a smooth torus, $M^m$, which is warped with a
cosine function to create the curved manifold, $M^m \times [0,D/2]$,
in which $C$ lies.  On the right we see the the set $U$ viewed as a subset
of a hemisphere created using the same warping function. }
   \label{fig-lem-hem}
\end{figure}

Observe that $F:[0,h]\times [0,D/2] \to M'$ defined by $F(s,t)=(x(s/h),t)$
is an isometric embedding of a region, $U$, in the standard round sphere, $S^2$,
of diameter $D$ into $M'$.  That is the metric on $U$ is
\be
dt^2 + \cos^2(t\pi/D) ds^2
\ee
and for any curve $\gamma:[a,b]\to U$ where $\gamma(u)=(\gamma_s(u), \gamma_t(u))$
we have, by (\ref{ineq-lem-hem}),
\begin{eqnarray*}
g''((F \circ \gamma)'(u),(F \circ \gamma)'(u))&=&
(dt^2 +  \cos^2(t\pi/D) g) ((F \circ \gamma)'(u),F \circ \gamma)'(u))\\
&=&
\gamma_t'(u)^2 +  \cos^2(t\pi/D) g ((x\circ \gamma_s)'(u/h),(x\circ \gamma_s)'(u/h))/h^2\\
&=&
\gamma_t'(u)^2 +  \cos^2(t\pi/D) |\gamma_s'(u)|^2 g (x'(u/h),x'(u/h))/h^2\\
&=&
\gamma_t'(u)^2 +  \cos^2(t\pi/D) |\gamma_s'(u)|^2 \textrm{ because } g(x',x')=h^2\\
&=&
(dt^2+ \cos^2(t\pi/D) ds^2) (\gamma'(u), \gamma'(u)).
\end{eqnarray*}
In particular, $L(F \circ x)=L(x)$ by (\ref{eq-lem-hem}).

Furthermore $C(s)\subset F(U)$.  So $F^{-1}\circ C$
is a curve in $S^2$ running between $F^{-1}(C(0))$ and $F^{-1}(C(1))$.
Thus
\be
L_{g''}(C) \ge d_{S^2}(F^{-1}(C(0)), F^{-1}(C(1)))= L(F^{-1}\circ x)
=d_M(x(0),x(1))
\ee
because $F^{-1}\circ x$ runs along a great circle in $S^2$ and has length $<D$.
This contradicts (\ref{no-no}).
\end{proof}

We may now apply Lemma~\ref{lem-hem} to prove Proposition~\ref{prop-hem}:

\begin{proof}
By Lemma~\ref{lem-hem}, (\ref{ineq-prop-hem}) and
(\ref{eq-prop-hem}), we see that  any geodesic, $\gamma:[0,1]\to M_i$,
of length $\le D_i$ satisfies (\ref{star-0}).  

Given $q_1, q_2\in M$, let $\gamma:[0,1]\to M'$ be
a length minimizing geodesic from $\varphi_1(q_1)$ to
$\varphi_2(q_2)$.  So
\be
\gamma(s)=(c(s), t(s))\in M \times [t_1, t_2].
\ee
For $f_i(t)=\cos((t-t_i)\pi/D_i)$
\begin{eqnarray*}
L_{g'}(\gamma)&=&\int_0^1 g'(\gamma'(s), \gamma'(s)) \, ds \\
&\ge &\int_0^1 \sqrt{t'(s)^2 + \max_{i=1,2}f_i^2(t(s)) g_i(c'(s),c'(s))} \, ds \\
&\ge &\int_0^1 \sqrt{t'(s)^2 + f_1^2(t(s)) g_1(c'(s),c'(s))} \, ds \\
&\ge &\int_0^1 \sqrt{t'(s)^2 + f_1^2(t(s)) g_1(\bar{c}'(s),\bar{c}'(s))} \, ds 
\end{eqnarray*}
where $\bar{c}$ is a length minimizing geodesic in $(M,g_1)$
from $c(0)$ to $c(1)$ parametrized proportional to arclength
of length $h=d_{g_1}(c(0),c(1))$.    Thus
\be
L_{g'}(\gamma)\ge \int_0^1 \sqrt{t'(s)^2 + f_1^2(t(s)) h^2 }\, ds. 
\ee
This integral is the length of a curve in a hemisphere of 
diameter $D_1$ running from a point $(0,h)$ on the equator
to a point $(|t_1-t_2|,0)$.   So it is greater than or equal to
the length of the third
side of a triangle opposite a right angle with legs of
length $d_{M_1}(q_1,q_2)$ and $|t_1-t_2|$.  Applying the
Spherical Law of Cosines rescaled we obtain 
\begin{eqnarray}
d_{M'}(\varphi_1(q_1), \varphi_2(q_2))&\ge&
\frac{D_1}{\pi}
\arccos\left( \cos\left(\frac{\pi d_{M_1}(q_1, q_2)}{D_1}\right)
\cos\left(\frac{\pi |t_1-t_2|}{D_1}\right) \right)\\
&\ge& d_{M_1}(q_1,q_2).
\end{eqnarray}
\end{proof}

\subsection{Estimating Distances between Manifolds}

The Gromov-Hausdorff distance between
a pair of metric spaces was estimated in terms of the Lipschitz distance
between them in \cite{Gromov-metric}.
In \cite{SorWen2}, the intrinsic flat distance between
a pair of integral current spaces was estimated in terms
of the Lipschitz distance between them.   Here we give a simple proof
estimating these distances between Riemannian manifolds
using explicit isometric embeddings into a common Riemannian
manifold.   Recall Definitions~\ref{metric-completion}
and~\ref{defn-positive-density}.

\begin{lem} \label{lem-squeeze-Z}
Suppose $M_1=(M,g_1)$ and $M_2=(M,g_2)$ are diffeomorphic
oriented precompact
Riemannian manifolds and suppose there exists $\epsilon>0$ such that
\be
g_1(V,V) < (1+\epsilon)^2 g_2(V,V) \textrm{ and }
g_2(V,V) < (1+\epsilon)^2 g_1(V,V) \qquad \forall \, V \in TM.
\ee
Then for any
\be
a_1> \frac{\arccos(1+\epsilon)^{-1} }{\pi}\diam(M_2)
\ee
and
\be
a_2> \frac{\arccos(1+\epsilon)^{-1}}{\pi} \diam(M_1),
\ee
there is a pair of isometric embeddings
$\varphi_i:M_i \to M'=\bar{M} \times [t_1, t_2]$ with a metric
as in Proposition~\ref{prop-hem}
where $t_2-t_1\ge \max\left\{a_1, a_2\right\}$.

Thus the Gromov-Hausdorff distance
between the metric completions is bounded,
\be \label{sq-Z-GH}
d_{GH}(\bar{M}_1, \bar{M}_2 ) \le a:=\max\left\{a_1, a_2\right\},
\ee
and the intrinsic flat and scalable intrinsic flat distances
between the settled completions are bounded,
\be\label{77}
d_{\mathcal{F}}(M'_1, M'_2) \le a\left(V_1+ V_2 + A_1+A_2\right),
\ee
\be\label{78}
d_{s\mathcal{F}}(M'_1, M'_2) \le
\left( a(V_1+ V_2)\right)^{1/(m+1)} +
\left( a (A_1+ A_2)\right)^{1/(m)}
\ee
where $V_i=\vol_m(M_i)$ and $A_i=\vol_{m-1}(\partial M_i)$.
\end{lem}

\begin{proof}
By our choice of $a_i$ we have
\be
g_1(V,V) >\cos^2(a_1\pi/diam(M_2)) g_2 (V,V) \qquad \forall V\in TM
\ee
and
\be
g_2(V,V) > \cos^2(a_2\pi/\diam(M_1)) g_1 (V,V) \qquad \forall V\in TM.
\ee
Applying Proposition~\ref{prop-hem} and setting $t_1=0$ and $t_2=a$,  we 
have isometric embeddings $\varphi_i: (M, g_i) \to (M',g')$
for any $g' $ satisfying (\ref{ineq-prop-hem}) and (\ref{eq-prop-hem}).  
In fact, the metric $g'$
on $M'$ can be chosen so that
\be \label{star-1}
g'(V,V) \le dt^2(V,V) +  g_1(V,V)+ g_2(V,V) \qquad \forall \, V \in TM'.
\ee

By Definition~\ref{defn-GH} we have,
\be
d_{GH}(\bar{M}_1, \bar{M}_2)
\le d^{M'}_H(\varphi_1(M_1), \varphi_2(M_2)).
\ee
For all $r>a$, $\varphi_1(M_1)\subset T_r(\varphi_2(M_2))$
and $\varphi_2(M_2)\subset T_r(\varphi_1(M_1))$,
because for all $p\in M$ we have
\be
d_{M'}(\varphi_1(p), \varphi_2(p)) = |t_2-t_1|=a.
\ee
By Definition~\ref{defn-H} we have (\ref{sq-Z-GH}).

Recall that to estimate the Intrinsic
flat Distance and scalable intrinsic flat distance
we must estimate volumes of a filling manifold and an 
excess boundary as
in (\ref{est-int-flat}) and (\ref{est-s-int-flat}).
Taking
$\nu$ to be the unit inward normal to $\partial M' \setminus (M_1\cup M_2)$
and applying the estimate on $g'$ given in (\ref{star-1}) we have
\begin{eqnarray*}
d_{\mathcal{F}}(M'_1, M'_2) &\le& \vol(M') + \vol (\partial M' \setminus (M_1\cup M_2))\\
&=& \int_{t_1}^{t_2}\int_M \mu_{g'}\,dt + \int_{t_1}^{t_2}\int_{\partial M} \nu \lrcorner \mu_{g'}\,dt\\
&\le & |t_2-t_1| (\vol(M_1)+ \vol(M_2)) + |t_2-t_1| (\vol(\partial M_1)+ \vol(\partial M_2)).
\end{eqnarray*}
and
\begin{eqnarray*}
d_{s\mathcal{F}}(M'_1, M'_2) &\le& \left(\vol(M')\right)^{1/(m+1)} +
\left(\vol (\partial M' \setminus (M_1\cup M_2))\right)^{1/m}\\
&=& \left(\int_{t_1}^{t_2}\int_M \mu_{g'}\,dt\right)^{1/(m+1)} +
\left(\int_{t_1}^{t_2}\int_{\partial M} \nu \lrcorner \mu_{g'}
\,dt\right)^{1/m}\\
&\le & \left(|t_2-t_1| ( \vol(M_1)+ \vol(M_2))\right)^{1/(m+1)}  \\
&&\qquad +
\left(|t_2-t_1| (\vol(\partial M_1)+ \vol(\partial M_2))\right)^{1/m}.
\end{eqnarray*}
\end{proof}

\subsection{Appending Regions without Smooth Approximations}

Now we examine pairs of precompact oriented
manifolds $(M_1, g_1)$ and
$(M_2, g_2)$ which are not diffeomorphic but have
diffeomorphic regions $U_i\subset M_i$.  That is, there
is a common smooth manifold with boundary $U$ and
diffeomorphisms $\psi_i: U \to U_i\subset M_i$.  Then we may
apply Proposition~\ref{prop-hem} and
Lemma~\ref{lem-squeeze-Z} to the regions $U_i$
to estimate the distances
between the metric and settled completions of the $M_i$.
Recall also Definitions~\ref{metric-completion}
and~\ref{defn-positive-density}.   Recall also the distinction
between the intrinsic length metric, $d_U$, and the 
restricted metric $d_M$, on a region $U\subset M$
and the corresponding diameters, $\diam_M(U)\le \diam_U(U)$,
in Definition~\ref{defn-d}.

\begin{thm} \label{thm-subdiffeo}
Suppose $M_1=(M,g_1)$ and $M_2=(M,g_2)$ are oriented
precompact Riemannian manifolds
with diffeomorphic subregions $U_i \subset M_i$ and
diffeomorphisms $\psi_i: U \to U_i$ such that
\be \label{thm-subdiffeo-1}
\psi_1^*g_1(V,V)
< (1+\epsilon)^2 \psi_2^*g_2(V,V) \qquad \forall \, V \in TU
\ee
and
\be \label{thm-subdiffeo-2}
\psi_2^*g_2(V,V) <
(1+\epsilon)^2 \psi_1^*g_1(V,V) \qquad \forall \, V \in TU.
\ee
Taking the extrinsic diameters,
\be \label{DU}
D_{U_i}= \sup\{\diam_{M_i}(W): \, W\textrm{ is a connected component of } U_i\} \le \diam(M_i),
\ee
we define a hemispherical width,
\be \label{thm-subdiffeo-3}
a>\frac{\arccos(1+\epsilon)^{-1} }{\pi}\max\{D_{U_1}, D_{U_2}\}.
\ee
Taking the difference in distances with respect to the outside manifolds,
\be \label{lambda}
\lambda=\sup_{x,y \in U}
|d_{M_1}(\psi_1(x),\psi_1(y))-d_{M_2}(\psi_2(x),\psi_2(y))|,
\ee
we define heights,
\be \label{thm-subdiffeo-4}
h =\sqrt{\lambda ( \max\{D_{U_1},D_{U_2}\} +\lambda/4 )\,}
\ee
and
\be \label{thm-subdiffeo-5}
\bar{h}= \max\{h,  \sqrt{\epsilon^2 + 2\epsilon} \; D_{U_1}, \sqrt{\epsilon^2 + 2\epsilon} \; D_{U_2} \}.
\ee
Then the Gromov-Hausdorff distance between the metric
completions is bounded,
\be \label{thm-subdiffeo-6}
d_{GH}(\bar{M}_1, \bar{M}_2 ) \le a + 2\bar{h} +
\max\left\{ d^{M_1}_H(U_1, M_1), d^{M_2}_H(U_2, M_2)\right\}
\ee
and the intrinsic flat distance between the
settled completions is bounded,
\begin{eqnarray*}
d_{\mathcal{F}}(M'_1, M'_2) &\le&
\left(2\bar{h} + a\right) \Big(
\vol_m(U_{1})+\vol_m(U_2)+\vol_{m-1}(\partial U_{1})+\vol_{m-1}(\partial U_{2})\Big)\\
&&+\vol_m(M_1\setminus U_1)+\vol_m(M_2\setminus U_2).
\end {eqnarray*}
and the scalable intrinsic flat distance is bounded,
\begin{eqnarray*}
d_{s\mathcal{F}}(M'_1, M'_2)
&\le&
\Big((\vol_m(U_{1}) +\vol_m(U_2))\left(\bar{h}+ a\right)
\Big)^{1/(m+1)}\\
&& + \,\,\Big(\, \big( 2\bar{h} + a\big)
(\vol_{m-1}(\partial U_{1})+\vol_{m-1}(\partial U_{2}))\\
&&\qquad+\vol_m(M_1\setminus U_1)+\vol_m(M_2\setminus U_2)
\,\Big)^{1/m}.
\end{eqnarray*}
\end{thm}

Figure~\ref{fig-defn-GH} may be viewed as an application of this theorem.
It should be noted that this theorem is an improvement on the
Bridge Method Lemma A.2 of \cite{SorWen2} in two respects.  First,
we allow $U_1$ and $U_2$ not isometric, and secondly we
loosen the diameter bounds of that method asking only for
control on the $\lambda$ defined here.

Recall in Definition~\ref{defn-d}, that two different metrics
are defined on a connected subdomain, $U\subset M$.  When
$U$ is also totally convex, these two metrics agree.   
Theorem~\ref{thm-subdiffeo} does not require the subdomains to
be connected or convex, and so the proof becomes quite 
difficult.  Before we prove this theorem we state and prove a special
case with stronger estimates.

\begin{thm} \label{thm-convex}
Suppose $M_1=(M,g_1)$ and $M_2=(M,g_2)$ are oriented Riemannian manifolds
with diffeomorphic totally convex subregions $U_i \subset M_i$ and
diffeomorphisms $\psi_i: U \to U_i$ such that
\be
\psi_1^*g_1(V,V)
< (1+\epsilon)^2 \psi_2^*g_2(V,V) \qquad \forall \, V \in TU
\ee
and
\be
\psi_2^*g_2(V,V) <
(1+\epsilon)^2 \psi_1^*g_1(V,V) \qquad \forall \, V \in TU.
\ee
Then for any
\be
a_1> \frac{\arccos(1+\epsilon)^{-1} }{\pi}\diam_{U_2}(U_2)
\ee
and
\be
a_2> \frac{\arccos(1+\epsilon)^{-1}}{\pi} \diam_{U_1}(U_1),
\ee
there is a pair of isometric embeddings $\varphi_i: U_i \to M'$
where $M'=U\times [t_1, t_2]$ where $t_2-t_1=\max\left\{a_1, a_2\right\}$
such that $\varphi_i(x)=(x,t_i)$.  Furthermore, these
isometric embeddings extend to isometric embeddings
$\varphi: M_i \to Z'$, where $Z'$ is a length metric space
defined by gluing $M_i$ to $M'$ along $U_i$.

In particular the Gromov-Hausdorff distance
between the metric completions is bounded,
\be
d_{GH}(\bar{M}_1, \bar{M}_2 ) \le \max\left\{a_1, a_2\right\} +
\max\left\{ d^{M_1}_H(U_1, M_1), d^{M_2}_H(U_2, M_2)\right\}
\ee
and the intrinsic flat distance
between the settled completions is bounded,
\begin{eqnarray*}
d_{\mathcal{F}}(M'_1, M'_2) &\le& \max\left\{a_1, a_2\right\}
\Big(\vol(U_1)+ \vol(U_2) + \vol(\partial U_1)+\vol(\partial U_2)\Big) \\
&&+ \vol(M_1\setminus U_1)
+ \vol(M_2 \setminus U_2),
\end {eqnarray*}
and the scalable intrinsic flat distance is bounded,
$$
d_{s\mathcal{F}}(M'_1, M'_2) \,\,\le\,\,
\Big(\,\max\left\{a_1, a_2\right\} (\vol(U_1)+ \vol(U_2))\,\Big)^{1/(m+1)}  \qquad\qquad\qquad\qquad
$$
$$
\qquad +\,\,\Big(\,\max\left\{a_1, a_2\right\} (\vol(\partial U_1)+ \vol(\partial U_2))
+ \vol(M_1\setminus U_1) + \vol(M_2\setminus U_2) \,\Big)^{1/(m)} .
$$
\end{thm}

\begin{proof}
The metric $g'$ on $M'$ is defined by applying 
Proposition~\ref{prop-hem}
and Lemma~\ref{lem-squeeze-Z} to the diffeomorphic regions,
$U_1$ and $U_2$; taking $D_i = \diam_{U_i}(U_i)$ as defined above, 
$\varphi_i: U_i \to M'$  are isometric embeddings.   
We can choose $g'$ satisfying
(\ref{star-1}).

We must verify that the $M_i$ isometrically embed into $Z'$
constructed as in the statement of the theorem.  To see this we
take any $x,y\in M_1$ and a shortest curve $C \subset Z'$ running between
$\varphi_1(x)$ and $\varphi_1(y)$.
If the curve never enters $\varphi_2(M_2\setminus U_2)$ then
$d_{M_1}(x,y)=d_{Z'}(\varphi_1(x),\varphi_1(y))$ by
Lemma~\ref{lem-squeeze-Z} and Lemma A.1 in the Appendix
of \cite{SorWen2} applied to $\varphi_1(M_1) \cup M' \subset Z'$.  If
the curve does enter $\varphi_2(M_2\setminus U_2)$ then
we have a length minimizing curve which leaves $\bar{U}_2$
contradicting the fact that it is convex.   The same argument may
be repeated to prove $\varphi_2: M_2 \to Z'$ is an isometric
embedding.

So now we may estimate the Gromov-Hausdorff distance
as in Remark~\ref{rmrk-Z-b}.
Let $r_i = d^{M_i}_H(M_i, U_i)$.  We claim
\be
d_{GH}(M_1, M_2) \le
d^{Z'}_H\big(\varphi_1(M_1), \varphi_2(M_2)\big) \le |t_1-t_2| + \max\left\{r_i\right\}.
\ee
Fix any $\delta>0$.
Then any point $p\in M_1$
has a point $q \in U_1$ such that $d(p,q)\le r_1 +\delta$.
Furthermore, $\varphi_1(q)=(q, t_1) \in M'\subset Z'$ so
\be
d_{Z'}(\varphi_1(q), (q,t_2) ) = |t_2-t_1|
\ee
and $(q,t_2)\subset \varphi_2(U_2) \subset \varphi_2(M_2)$.
Thus
\be
\varphi_1(M_1) \subset T_{r_1+\delta+|t_2-t_1|}(\varphi(M_2) )
\ee
and similarly
\be
\varphi_2(M_2) \subset T_{r_2+\delta+|t_2-t_1|}(\varphi(M_1) ).
\ee
The claim follows by taking $\delta\to 0$.

We bound the intrinsic flat distance as in Remark~\ref{rmrk-Z-c}
taking $M'$ to be the filling
manifold with the metric $g'$ defined in Lemma~\ref{lem-squeeze-Z}
satisfying (\ref{star-1}).   
We apply the same estimates as in
Lemma~\ref{lem-squeeze-Z} to bound the volumes of these
regions, only now we add in the additional volume terms 
coming from the additional components of the excess boundary
$M_i\setminus U_i$.   

We bound the scalable intrinsic flat distance 
as in Remark~\ref{rmrk-Z-s}.  Again we include
the additional components of the excess boundary but insert them
into the summand with an exponent of $1/m$ since these are
$m$ dimensional boundary regions and the scalable flat distance
is 1 dimensional.
\end{proof}

We now prove Theorem~\ref{thm-subdiffeo}.  To prove this
theorem we adapt the proof of the convex case and the proof
of Lemma A.2 in \cite{SorWen2}.   It is essential to possibly
push the two manifolds further apart than required simply to isometrically
embed the $U_i$ into $M'$ as a short cut for a path between
points in $\varphi_1(M_1)$ might be found within
$\varphi_2(M_2\setminus U_2)$.

\begin{proof}
For each corresponding pair of connected components $U_{\alpha,i}$
of $U_i$, we create a hemispherically defined filling bridge
$M'_\alpha$ diffeomorphic to $U_{\alpha_i,i}\times [0,a]$ with metric $g'_\alpha$ satisfying (\ref{star-0})
by applying Proposition~\ref{prop-hem} and Lemma~\ref{lem-squeeze-Z} 
using the $a_i=a_i(\alpha)$ defined there for that particular connected
component, $U_\alpha$ and $D_i=D_{U_i}$.  Observe that all $a_i \le a$,
so we may take $|t_1-t_2|=a$ for all the connected
components.   Any minimal geodesic $\gamma: [0,1]\to U_{\alpha,i}$
of length $\le D_{U_i}\le \diam_{M_i}(U_i)$ satisfies (\ref{star-0}).

We take the
disjoint unions of these bridges to be $M'$.  So it has a metric
$g'$ satisfying (\ref{star-1}).   Observe that
the boundary of $M'$ is $(U,g_1) \cup (U, g_2) \cup (\partial U \times [0 , a] , g')$.  So that
\begin{eqnarray}\label{sdB}
\vol_m(M')&=&\sum_\alpha \vol_m(M_\alpha')\\
& \le& \sum_\alpha  a(\vol_m(U_{\alpha,1})+\vol_m(U_{\alpha,2}))\\
& \le& a (\vol_m(U_{1})+\vol_m(U_{2}))
\end{eqnarray}
and
\be\label{sdA}
\vol_{m}\left(\partial M'\setminus (\varphi_1(U_1)\cup
\varphi_2(U_2)\right) 
\le a \; (\vol_{m-1}(\partial U_1) + \vol_{m-1}(\partial U_2))
\ee
as in Lemma~\ref{lem-squeeze-Z}.

We cannot directly glue $M_i$ to $M'$ and obtain an isometric embedding
because our regions are not convex.  
On either end of the filling bridges, we glue isometric products
$U_{\alpha} \times [0, \bar{h}]$ with metric $dt^2 +g_i$,
so that all the bridges are extended
by an equal length on either side.  This creates a Lipschitz manifold,
\be
M''=(U_1\times[0,\bar{h}])
\disjointunion_{U_1} M' \disjointunion_{U_2} (U_2 \times [0,\bar{h}]).
\ee
We then define $\varphi_i: U_i \to M''$ such that
\begin{eqnarray}
\varphi_1(x)&=&(x,0)\in U_1\times [0, \bar{h}]\\
\varphi_2(x)&=&(x,\bar{h})\in U_2\times [0, \bar{h}]
\end{eqnarray}
as in Figure~\ref{fig-subdiffeo-1}.
Then by (\ref{sdB}) and (\ref{sdA}), we have
\begin{eqnarray}\label{subdiffeo-B}
\vol_{m+1}(M'')&=& \vol_{m+1}(M') + \bar{h} (\vol_m(U_1) +\vol_m(U_2))\\
& \le& (a+ \bar{h}) (\vol_m(U_{1})+\vol_m(U_{2}))
\end{eqnarray}
and $\vol_{m}\left(\partial M''\setminus (\varphi_1(U_1)\cup
\varphi_2(U_2)\right)=$
\begin{eqnarray} \label{subdiffeo-A}
\qquad
&=&
\vol_{m}\left(\partial M'\setminus (\varphi_1(U_1)\cup
\varphi_2(U_2)\right)+
 \bar{h} (\vol_{m-1}(\partial U_1) +\vol_{m-1}(\partial U_2))\\
 &\le& (a+\bar{h}) (\vol_{m-1}(\partial U_1) + \vol_{m-1}(\partial U_2)).
\end{eqnarray}

Finally we glue $M_1$ and $M_2$
to the far ends of $M''$ along $\varphi_i(U_i)$
to create a connected length space
\be
Z = \bar{M}_1 \disjointunion_{U_1} M'' \disjointunion_{U_2} \bar{M}_2
\ee
where distances in $Z$ are defined by taking the infimum of
lengths of curves as usual.   See Figure~\ref{fig-subdiffeo-1}.
We will refer to each connected component, $M''_\alpha$ of $M''$
as the filling bridge corresponding to $U_\alpha$.

\begin{figure}[h] 
   \centering
   \includegraphics[width=5in]{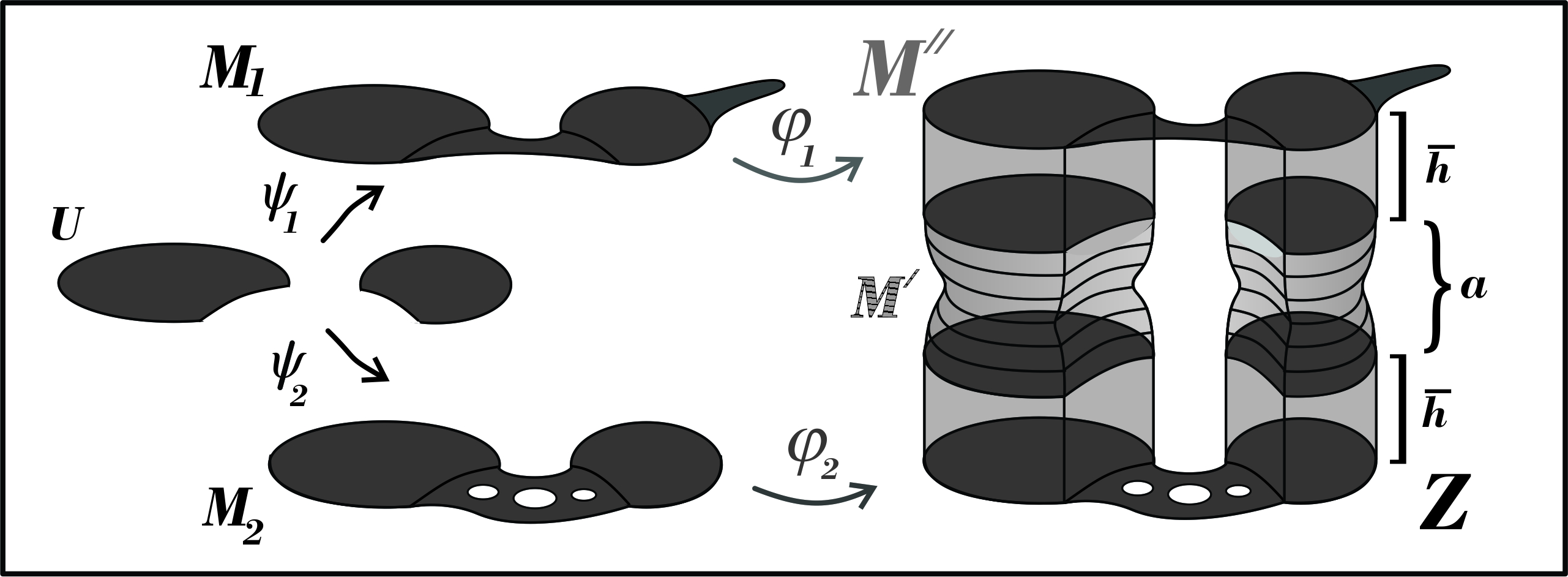}
   \caption{Creating $Z$ for Theorem~\ref{thm-subdiffeo}.}
   \label{fig-subdiffeo-1}
\end{figure}

We must prove that $\varphi_1: M_1 \to Z$ mapping $M_1$ into
its copy in $Z$ is an isometric embedding.  To see this we
take any $x,y\in M_1$ and a shortest curve $C\subset Z$ running between
$\varphi_1(x)$ and $\varphi_1(y)$.  As in the convex proof, our
only concern is the possibility that
$C$ passes into $\varphi_2(M_2\setminus U_2)$.

If the minimizing curve never crosses a filling bridge, then we
claim it has the same length as
a curve in $\varphi_1(M_1)$.   To see this, we take any
$s_2>s_1\in [0,1]$ such that $C(s_1), C(s_2)\in \varphi_1(M_1)$
and $C((s_1,s_2)) \subset Z\setminus \varphi_1(M_1)$.
Since $C$ is assumed not to cross a bridge (not to enter
$\varphi_2(M_2)$, then 
$C(s_1)=\varphi(x_1)$ and  $C(s_2)=\varphi(x_2)$ 
where $x_1,x_2$ lie in the same connected component, $U_{\alpha,1}$,
of $U_1$.   Since  $C([s_1,s_2])$ is a minimizing curve it has  length 
\be
\le d_{M''}(C(s_1), C(s_2))\le d_{M_1}(x_1,x_2) 
\le \diam_{M_1}(U_{\alpha,1})\le D_{U_1}.
\ee
By (\ref{star-0}), a minimal geodesic from $x_1$ to $x_2$
lying in $U_{\alpha,1}$ has the same length as $C([s_1,s_2])$.
So we may replace this segment of $C$ with the image
of this minimal geodesic.

On the other hand, if the minimizing curve crosses a 
filling bridge all the way to $\varphi_2(M_2)$, then we may carefully apply
the choice of $\bar{h}$ to reach a contradiction as in the left hand side of
Figure~\ref{fig-subdiffeo-2}.   We define the following
points $0\le t_1 < t_2 \le t_3 < t_4 \le 1$ such that
\begin{eqnarray}
t_1&=&\inf\{ t: \,\, C(t) \in Cl(Z \setminus \varphi_1(M_1))\},\\
t_2&=&\min\{ t: \,\, C(t) \in \varphi_2(M_2)\},\\
t_4&=&\min\{ t>t_2: \,\, C(t) \in \varphi_1(M_1)\}, \\
t_3&=&\max\{ t \in [t_2, t_4): \,\, C(t) \in \varphi_2(M_2)\}
\end{eqnarray}
so that $C([t_1, t_2])$ and $C([t_3, t_4])$ are geodesic
segments lying within filling bridges:
\be
C([t_1,t_2])\subset M''_{\alpha_{1,2}} \qquad
C([t_3,t_4]) \subset M''_{\alpha_{3,4}}.
\ee
\begin{figure}[h] 
   \centering
   \includegraphics[width=5in]{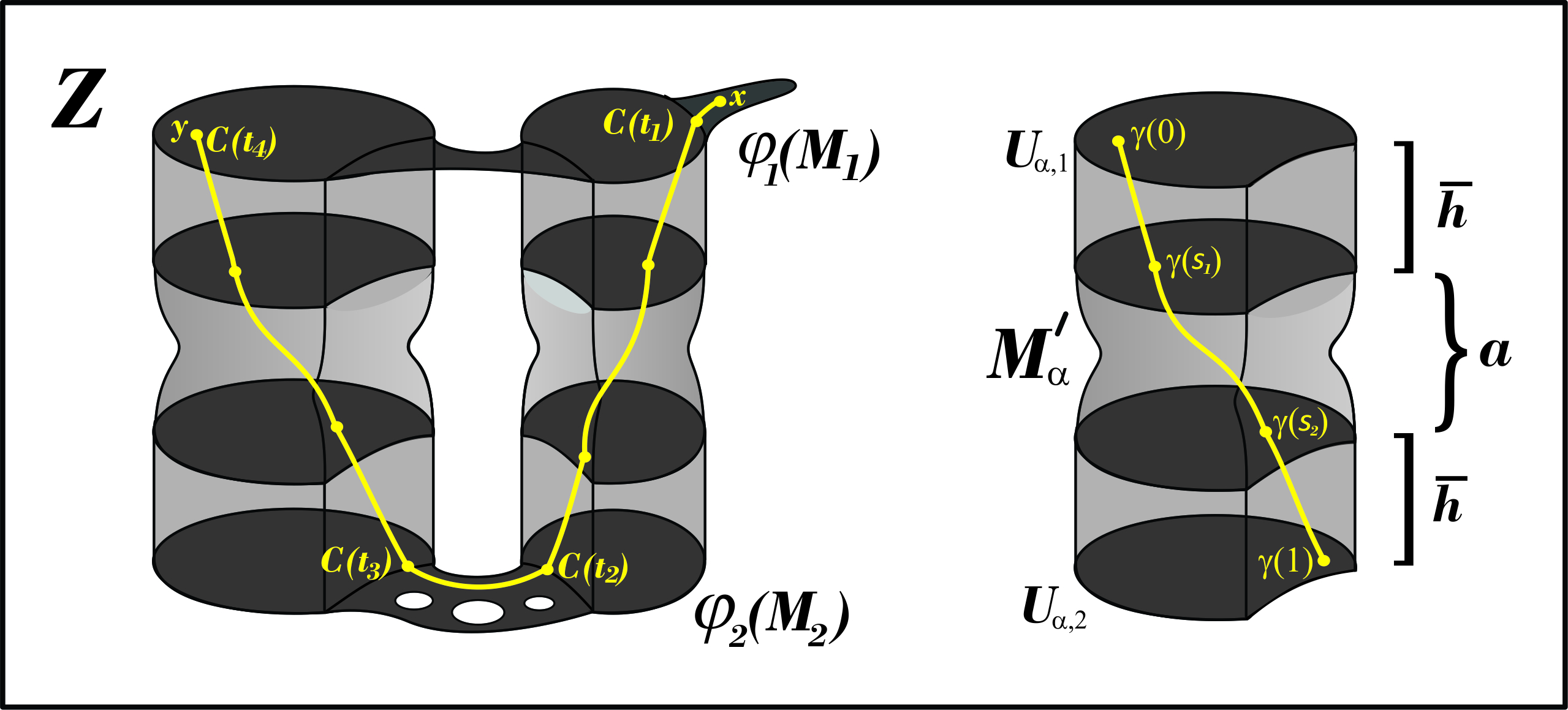}
   \caption{Why length minimizing curves cannot cross bridges
   in the proof of Theorem~\ref{thm-subdiffeo}.}
   \label{fig-subdiffeo-2}
\end{figure}

Observe that there are points $p_1, p_4\in U$
and $p_2, p_3 \in \partial(U)$ such that
\begin{eqnarray}
\varphi_1(\psi_1(p_1))=C(t_1)&&\varphi_1(\psi_1(p_4))=C(t_4)\\
\varphi_2(\psi_2(p_2))=C(t_2)&&\varphi_2(\psi_2(p_3))=C(t_3).
\end{eqnarray}
Observe that since $C([t_2,t_3])\subset \varphi_2(M_2)$
we know the length of this segment is
\be\label{rough-2}
d_{Z}(C(t_2), C(t_3)) =d_{M_2}(\psi_2(p_2),\psi_2(p_3)) \ge d_{M_1}(\psi_1(p_2),\psi_1(p_3))-\lambda
\ee
by the definition of $\lambda$ in (\ref{lambda}).

We claim that the lengths of the other segments are
\be \label{will-prove-1}
d_{M''}(C(t_1),C(t_2)) > \sqrt{ d_{M_1}(p_1,p_2)^2 + h^2\,}
\ee
and
\be \label{will-prove-2}
d_{M''}(C(t_3),C(t_4)) > \sqrt{d_{M_1}(p_3,p_4)^2 + h^2\,}.
\ee
Once we prove this claim, we see
that by the definition of $h$
we have
\begin{eqnarray}
d_{M''}(C(t_1),C(t_2)) &>& d_{M_1}(p_1,p_2) + \lambda/2,\\
d_{M''}(C(t_3),C(t_4)) &>& d_{M_1}(p_3,p_4) + \lambda/2.
\end{eqnarray}
This combined with (\ref{rough-2}) implies that
\begin{eqnarray*}
L(C([t_1,t_4])&=& d_{M''}(C(t_1), C(t_2))+
d_{Z}(C(t_2), C(t_3))+ d_{M''}(C(t_3), C(t_4))\\
&>& d_{M_1}(p_1,p_2) + \lambda/2+ d_{M_1}(p_2, p_3)-\lambda+
d_{M_1}(p_3,p_4) + \lambda/2\\
&\ge& d_{M_1}(p_1,p_2) + d_{M_1}(p_2, p_3) + 
d_{M_1}(p_3,p_4) \\
&\ge & d_{M_1}(p_1, p_4) =d_{\varphi_1(M_1)}(C(t_1), C(t_4))\\
&\ge& d_{Z}(C(t_1), C(t_4)),
\end{eqnarray*}
which is a contradicts the fact that $C$ was minimizing.

So we need only prove our claim
in (\ref{will-prove-1}) and (\ref{will-prove-2}) to see that
$\varphi_1: M_1 \to Z$ is an isometric embedding.  This claim concerns
a minimizing geodesic lying in a single connected component
of the filling bridges,
\be
\gamma:[0,1]\to (U_{\alpha,1}\times[0, \bar{h}])
\disjointunion_{U_{\alpha,1}}
M_\alpha' \disjointunion_{U_{\alpha,2}}
(U_{\alpha,2}\times[0, \bar{h}])
\ee
such that
\be
\gamma(0)=(q_0,0)\in U_{\alpha,1}\times \{0\}\subset \varphi_1(M_1)
\ee
and
\be
\gamma(1)=(q_3,\bar{h})\in U_{\alpha,2}\times \{\bar{h}\}\subset \varphi_2(M_2).
\ee
Consult the right hand side of Figure~\ref{fig-subdiffeo-2}.
Let $0<s_1<s_2<1$ be chosen so that
\be
\gamma(s_1)=(q_1, \bar{h})\in U_{\alpha,1}\times \{\bar{h}\}\subset \partial M_\alpha'
\ee
and
\be
\gamma(s_2)=(q_2, 0)\in U_{\alpha,2}\times \{0\}\subset \partial M_\alpha'.
\ee
Then by (\ref{extra-edges}), we have
\be
L_{g'}(\gamma([s_1,s_2]) = d_{g'}(\gamma(s_1), \gamma(s_2)) \ge d_{g_1}(q_1,q_2),
\ee
so that
\begin{eqnarray*}
L(\gamma)&=& d_{Z}(\gamma(0), \gamma(s_1))+
d_Z(\gamma(s_1), \gamma(s_2)) + d_Z(\gamma(s_2), \gamma(1))\\
&=& \sqrt{d_{U_1}(q_0,q_1)^2 + \bar{h}^2} +
d_{U_1}(q_1,q_2)
+ \sqrt{d_{U_2}(q_2,q_3)^2 + \bar{h}^2} \\
&\ge& \sqrt{d_{U_1}(q_0,q_1)^2 + h^2} +
d_{U_1}(q_1,q_2)\\
&&+ \qquad \sqrt{d_{U_1}(q_2,q_3)^2/(1+\epsilon)^2 + (\epsilon^2 + 2 \epsilon) D_{U_1}^2} \\
&\ge& \sqrt{d_{M_1}(q_0,q_1)^2 + h^2} +
d_{M_1}(q_1,q_2)\\
&&+ \qquad \sqrt{d_{M_1}(q_2,q_3)^2/(1+\epsilon)^2 + (\epsilon^2 + 2 \epsilon)
d_{M_1}(q_3,q_4)^2} \\
&\ge& \sqrt{d_{M_1}(q_0,q_1)^2 + h^2} +
d_{M_1}(q_1,q_2)
+ \sqrt{d_{M_1}(q_2,q_3)^2} \\
& > & \sqrt{(d_{M_1}(q_0,q_1)+d_{M_1}(q_1,q_2) + d_{M_1}(q_2, q_3))^2 + h^2}\\
&\ge & \sqrt{ d_{M_1}(q_0,q_3)^2 + h^2}.
\end{eqnarray*}
This gives us (\ref{will-prove-1}) and (\ref{will-prove-2}).
Thus we have proven $\varphi_1: M_1 \to Z$ is an isometric
embedding and the same follows for $\varphi_2:M_2 \to Z$.

So now we may estimate the Gromov-Hausdorff distance:
Let $r_i = d^{M_i}_H(M_i, U_i)$.  We claim
\be
d_{GH}(M_1, M_2) \le
d^{Z}_H(\varphi_1(M_1), \varphi_2(M_2) \le
\bar{h}+\bar{h}+ a + \max\left\{r_i\right\}.
\ee
Fix any $\delta>0$.
Then any point $p\in M_1$
has a point $q \in U_1$ such that $d(p,q)\le r_1 +\delta$.
Furthermore,
\be
d_{Z}(\varphi_1(q), \varphi_2(q) ) \le a + \bar{h}+\bar{h}
\ee
and $\varphi_2(q)\subset \varphi_2(U_2) \subset \varphi_2(M_2)$.
Thus
\be
\varphi_1(M_1) \subset T_{r_1+\delta+a+\bar{h} + \bar{h}}(\varphi(M_2) ).
\ee
and similarly
\be
\varphi_2(M_2) \subset T_{r_2+\delta+a+\bar{h} + \bar{h}}(\varphi(M_1) ).
\ee
The claim follows by taking $\delta\to 0$.

To bound the intrinsic flat distance and scalable intrinsic flat
distance, we take $B^{m+1}=M''$ to be the filling manifold and
then the excess boundary is
\be
A^m=\varphi_1(M_1\setminus U_1) \cup\varphi_2(M_2\setminus U_2) \cup
\partial M'' \setminus (\varphi_1(U_1)\cup \varphi_2(U_2))
\ee
so that with appropriate orientations we have
\be
\int_{\varphi_1(M_1)}\omega-\int_{\varphi_2(M_2)}\omega
= \int_{B^{m+1}} d\omega + \int_{A^m} \omega.
\ee
The volumes of these manifolds have been computed in
(\ref{subdiffeo-A}) and (\ref{subdiffeo-B}).  So as in
Remark~\ref{rmrk-Z-c} we have
\begin{eqnarray*}
d_{\mathcal{F}}(M_1, M_2) &\le&
\vol_m(U_{1})\left( \bar{h} + a\right) +\vol_m(U_2)\left( \bar{h}+ a\right)\\
&& + \left( \bar{h} + a\right)\vol_{m-1}(\partial U_{1})+
\left( \bar{h}+ a\right)\vol_{m-1}(\partial U_{2})\\
&&+\vol_m(M_1\setminus U_1)+\vol_m(M_2\setminus U_2).
\end {eqnarray*}
The scalable intrinsic flat distance is bounded
as in Remark~\ref{rmrk-Z-s} so that we have
\begin{eqnarray*}
d_{s\mathcal{F}}(M_1, M_2)
&\le&
\Big(\vol_m(U_{1})\left( \bar{h}+ a\right) +\vol_m(U_2)\left(\bar{h} + a\right)
\Big)^{1/(m+1)}\\
&& + \Big( \left( \bar{h} + a\right)\vol_{m-1}(\partial U_{1})+
\left( \bar{h} + a\right)\vol_{m-1}(\partial U_{2})\\
&&+\vol_m(M_1\setminus U_1)+\vol_m(M_2\setminus U_2)\Big)^{1/m}.
\end{eqnarray*}
\end{proof}

\section{Intrinsic Flat Limits}\label{sect-IF}

In this section we examine sequences of Riemannian
manifolds which converge smoothly away from singular
sets and their intrinsic flat limits proving Theorem~\ref{codim-thm}. 
This theorem will be shown to be consequences of the following more
powerful theorem which requires a condition on the embeddings
of the exhaustion in the manifold:

\begin{defn} \label{well-embedded}
Given a sequence of Riemannian manifolds
$M_i=(M,g_i)$ and an open subset, $U\subset M$,
a connected precompact exhaustion, $W_j$, of $U$ 
satisfying (\ref{defn-precompact-exhaustion})
is
{\em uniformly well embedded} if  \footnote{\textcolor{blue}{The limits have been reordered to match what we need to prove Theorem 5.2 with its original proof.}}
\be\label{lambda-ijk-5000}
\lambda_{i,j,k}= \sup_{x,y\in W_j} |d_{(W_{k}, g_i)}(x,y)- d_{(M,g_i)}(x,y)|   
\ee   
has
\textcolor{blue}{
\be\label{lambda-ijk-2-2020}
\limsup_{j\to \infty}\limsup_{k\to \infty} \limsup_{i \to \infty} \lambda_{i,j,k}=0.
\ee
and thus a uniform upper bound
\be \label{lambda-ijk-00-2020}
 \lambda_{i,j,k}
\le \lambda_0 < \infty
\ee
} 
\end{defn}

\begin{thm}\label{flat-to-settled}
Let $M_i=(M,g_i)$ be a sequence of compact oriented Riemannian manifolds
such that
there is a closed subset, $S$, and a 
uniformly well embedded connected precompact exhaustion,
$W_j$, of $M\setminus S$ satisfying (\ref{defn-precompact-exhaustion})
such that $g_i$ converge smoothly to $g_\infty$ on each $W_j$
with
\be\label{diam-3}
\diam_{M_i}(W_j) \le D_0 \qquad \forall i\ge j, 
\ee
\be \label{area-3}
\vol_{g_i}(\partial W_j) \le A_0
\ee
and
\be \label{not-vol-3}
\vol_{g_i}(M\setminus W_j) \le V_j \textrm{ where } \lim_{j\to\infty}V_j=0.
\ee
Then
\be
\lim_{j\to \infty} d_{\mathcal{F}}(M_j', N')=0
\ee
where $N'$ is the settled completion
of $N=(M\setminus S, g_\infty)$.   \footnote{\textcolor{blue}{With the correction to Definition 5.1 this theorem is
correct as originally stated but the proof is fixed within.}}
\end{thm}

In the first subsection, we prove a technical proposition
demonstrating that the intrinsic flat limit of a connected precompact exhaustion
of an open set in a fixed Riemannian manifold is the metric
completion of that open set [Proposition~\ref{flat-exhaustion}].  
This theorem is shown to be false for Gromov-Hausdorff limits 
[Example~\ref{ex-spiralling}].

The second subsection, we complete the proof of 
Theorem~\ref{flat-to-settled} applying Proposition~\ref{flat-exhaustion}
and Theorem~\ref{thm-subdiffeo}.  

The third subsection contains a proof of
Lemma~\ref{codim-2-lambda}
concerning manifolds with singular sets of codimension two. 
This final lemma combined with Theorem~\ref{flat-to-settled} proves
Theorem~\ref{codim-thm}.  \textcolor{blue}{This third subsection is now changed
and will be replaced by work in the appendix.}

\begin{rmrk} \label{flat-necessity}
In Example~\ref{ex-not-connected} we see that it
is necessary to assume that the exhaustion is
connected in Theorem~\ref{flat-to-settled}.  
The excess volume bound in (\ref{not-vol-2}) is shown to
be necessary in Example~\ref{ex-to-hemisphere} and
Example~\ref{ex-cap-cyl}, which has no intrinsic flat limit.
The uniform bound on the boundary volumes, (\ref{area-2}),
is seen to be necessary in Example~\ref{ex-flamenco}.
All these examples have codimension 2 singular
sets and show the necessity of these hypothesis
for Theorem~\ref{codim-thm} as well.
The uniform embeddedness hypothesis of Theorem~\ref{flat-to-settled}
and the codimension two condition of Theorem~\ref{codim-thm}
are seen to be necessary for their respective theorems
in Example~\ref{ex-to-torus-square}.
\end{rmrk}

\subsection{Creating Spaces from Exhaustions}

In this section we examine the construction of the
limit space from a sequence of precompact open sets.
One may view this section as a technical subsection.
Recall that a connected precompact exhaustion of a domain
satisfies (\ref{defn-precompact-exhaustion}).

\begin{prop}\label{flat-exhaustion}
Let $W_j$ be a connected precompact exhaustion of a Riemannian
manifold, $N$, with fixed Riemannian metric, $g_N$.
If we assume that $\diam(N) \le D_0$,
$\vol(W_j) \le \vol(N) \le V_0$ and
$\vol(\partial W_j) \le A_0$ 
then the settled completion
$N' \subset \bar{N}$  satisfies
\be
\lim_{j\to\infty} d_{\mathcal{F}}\big( (W'_j, d_{W'_j}), ({N}', d_{{N}'}) \big) =0,
\ee
where $d_{W_j}$ is the induced length metric on $W_j$
defined by the Riemannian metric $g_N$ and $W_j'$ is the settled
completion of $W_j$ with respect to $d_{W_j}$. \footnote{\textcolor{blue}{This proposition is fine
because all distances are measured using lengths with respect to the same $g_N$.  A few
typos in the proof are fixed.   It now works with new reordering of limits
in $\lambda_{i,j,k}$.}}
\end{prop}

The connectedness is essential to this theorem as can be seen
in Example~\ref{ex-not-connected}.
Interestingly, one does not obtain Gromov-Hausdorff
convergence under these conditions.  There need not even 
exist a Gromov-Hausdorff limit of $(\bar{W}_j, d_{W_j})$.
See Example~\ref{ex-spiralling} below.   

\begin{proof}
We first verify that we can apply Theorem~\ref{thm-subdiffeo} with $M_1=W_{k}$ and $M_2={N}$ and $U_1=W_i \subset W_{k}$ for $i<k$ and
$U_2=W_i \subset {N}$.   Note that $\epsilon=0$ and the
hemispherical width $a$ can be taken to be $0$ because
$U_i$ have the same Riemannian metric, $g_N$.

We claim
\be \label{edge-volume}
\lim_{j\to\infty} \vol(N\setminus W_j) =0.
\ee
Since $N$ is an open manifold of finite volume
\be
\vol(N) =\sum_{k=1}^\infty \vol(W_k \setminus W_{k-1}),
\ee
so 
\be
\lim_{j\to \infty} \sum_{k=j+1}^\infty \vol(W_k \setminus W_{k-1})=0.
\ee
However
\be
\vol(N\setminus W_j) =\sum_{k=j+1}^\infty \vol(W_k \setminus W_{k-1})
\ee
so we have our claim.   

Let $k>i$ and let
\be\label{lambda-ik-1}
\lambda_{i,k}= \sup_{x,y\in W_i} |d_{W_{k}}(x,y)- d_{{N}}(x,y) |.
\ee
Then
\be \label{DU1}
D_{U_1}=\diam_{W_k}(W_i)\le \diam_{{N}}{W_i} +\lambda_{i,k} \le
D_0+\lambda_{i,k}.
\ee
and,
\be \label{DU2}
D_{U_2}=\diam_{{N}}{W_i}  \le D_0.
\ee

We claim that for fixed $i$,
\be \label{lim-lambda-1}
\lim_{k\to \infty} \lambda_{i,k} =0.
\ee
First note that $\lambda_{i,k}$ is decreasing in $k$ because
\be
d_{W_{k}}(x,y)\ge d_{W_{k+1}}(x,y) \ge d_{{N}}(x,y).
\ee
If the limit is not zero in (\ref{lim-lambda-1}) then let
\be\label{eps-prime}
\epsilon'=\inf_k \lambda_{i,k}>0.
\ee
Since $\bar{W}_i$ is compact, there exists $x_{i,k}, y_{i,k} \subset \bar{W}_i$
achieving the supremum in (\ref{lambda-ik-1}).   Taking $k$ to
infinity, a subsequence converges to $x_i, y_i\subset \bar{W}_i$
with respect to $d_{\bar{W}_i}$.
Let $\gamma_i \subset N$ be a curve from $x_i$ to $y_i$
such that
\be
L(\gamma_i) \le d_{N}(x_i,y_i) +\epsilon'/5.
\ee
Since $W_k$ exhaust $N$, there exists $N_{\epsilon'}$ sufficiently
large that
\be
\gamma_i \subset W_k \qquad \forall k\ge N_{\epsilon'}.
\ee
Thus
\be
d_{W_k}(x_i,y_i) \le d_{N}(x_i,y_i) +\epsilon'/5
\qquad \forall k\ge N_{\epsilon'}.
\ee
Now take $k$ from the subsequence sufficiently large
that we have
\begin{eqnarray}
d_{\bar{W}_k}(x_{i,k},x_i) &\le& d_{\bar{W}_i}(x_{i,k},x_i) <\epsilon'/5\\
d_{\bar{W}_k}(y_{i,k},y_i) &\le& d_{\bar{W}_i}(y_{i,k},y_i) <\epsilon'/5.
\end{eqnarray}
Thus
\begin{eqnarray*} 
d_{W_k}(x_{i,k},y_{i,k}) 
&\le & d_{W_k}(x_{i_k}, x_i) + d_{W_k}(x_i,y_i) + d_{W_k}(y_i, y_{i_k})\\
&< & d_N(x_i, y_i) + 3\epsilon'/5\\
&\le & d_{N}(x_{i_k}, x_i) + d_{N}(x_i,y_i) + d_{N}(y_i, y_{i_k})+
3\epsilon'/5\\
&\le & d_{W_k}(x_{i_k}, x_i) + d_{N}(x_i,y_i) + d_{W_k}(y_i, y_{i_k})+
3\epsilon'/5\\
&\le & d_{W_k}(x_{i_k}, x_i) + d_{N}(x_i,y_i) + d_{W_k}(y_i, y_{i_k})+
3\epsilon'/5\\
&<&\textcolor{blue}{d_{N}(x_i,y_i)}+ 5\epsilon'/5=\textcolor{blue}{d_{N}(x_i,y_i)}+\epsilon'.
\end{eqnarray*}
Since $d_N(x_{i,k},y_{i,k}) \le d_{W_k}(x_{i,k},y_{i,k})$, we have
\be
 \lambda_{i,k} < d_{W_k}(x_{i,k},y_{i,k})-d_N(x_{i,k}, y_{i,k}) \textcolor{blue}{\,\,<\,\,}
 \epsilon'
 \ee
  which contradicts (\ref{eps-prime}).

By (\ref{DU1}), (\ref{DU2}) and (\ref{lim-lambda-1}), we know that
for fixed $i$,
\be
\lim_{k\to \infty} \bar{h}_{i,k} =0
\ee
where $\bar{h}_{i,k}$ is defined as in
(\ref{thm-subdiffeo-4})-(\ref{thm-subdiffeo-5}) with $\lambda=\lambda_{i,j}$
and $\epsilon=0$.

By Theorem~\ref{thm-subdiffeo}, 
the intrinsic flat distance is bounded,
\begin{eqnarray*}
d_{\mathcal{F}}(N', W_j) &\le&
\left(\bar{h}_{i,j}\right) \left(
2\vol(W_i)+2 \vol(\partial W_i) \right)
+\vol(N\setminus W_i)+\vol(W_j\setminus W_i)\\
&\le&\left(\bar{h}_{i,j}\right) \left(2 V_0+2A_0 \right)
+\vol(N\setminus W_i)+\vol(N\setminus W_i)
\end {eqnarray*}
By (\ref{edge-volume}), for any $\epsilon">0$ there
exists $i$ sufficiently large that
\be
d_{\mathcal{F}}(N', W_j) \le \left(\bar{h}_{i,j}\right) \left(2 V_0+2A_0 \right)
+ \epsilon".
\ee
Fixing that value for $i$, we now take $j \to \infty$,
\be
\lim_{j\to \infty} d_{\mathcal{F}}(N', W_j) <\epsilon".
\ee
We have the theorem as stated.
\end{proof}

\begin{example}\label{ex-spiralling}
In Figure~\ref{fig-spiralling} we see that a connected precompact exhaustion
$W_j$ of a standard flat two dimensional
torus satisfying $\vol(W_j) \le \vol(N) \le V_0$ and
\be 
\lim_{j\to\infty} \vol(N\setminus W_j) =0.
\ee
Observe that $(\bar{W}_j, d_{W_j})$
need not have a Gromov-Hausdorff limit because balls of radius $1/2$
about the tips of the arms measured with respect to the intrinsic
length metric $d_{W_j}$ are disjoint and so the number of disjoint 
balls of radius $1/2$ is unbounded.  According to the
converse of Gromov Compactness
Theorem, the number of disjoint balls in a sequence of
compact metric spaces converging to a compact metric space
is uniformly bounded above, so this
sequence cannot converge \cite{Gromov-metric}.

\begin{figure}[h] 
   \centering
   \includegraphics[width=5in]{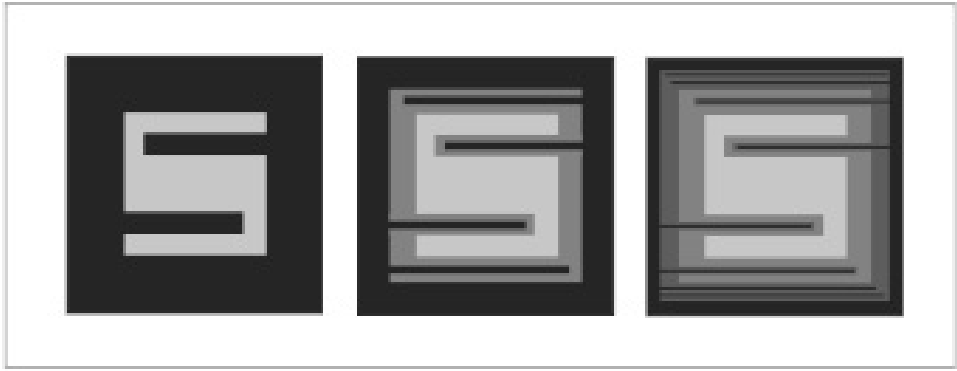}
   \caption{$W_1$ with two arms is depicted in white on a black $T^2$.
   $W_2$ with four arms is depicted in light grey containing $W_1$.  
   $W_3$ with eight
   arms is depicted in darker grey containing $W_2$ and $W_1$.}
   \label{fig-spiralling}
\end{figure}
To find an example which also satisfies $\vol(\partial W_j) \le A_0$,
we may construct a connected precompact exhaustion of a standard
flat three dimensional torus where the arms are thin tubular neighborhoods
of curves so that their lengths are still long enough to have disjoint
balls but the areas of the boundaries of the arms are arbitrarily small.
\end{example}

\subsection{Proof of Theorem~\ref{flat-to-settled}}

In this subsection we prove Theorem~\ref{flat-to-settled}.
Keep in mind Remark~\ref{flat-necessity}.   First we
prove a short lemma which will be applied here and
elsewhere:

\begin{lem} \label{lem-vol-vol}
Let $M_i=(M,g_i)$ be a sequence of compact Riemannian manifolds
such that
there is a closed subset, $S$, and a connected precompact exhaustion,
$W_j$, of $M\setminus S$ satisfying (\ref{defn-precompact-exhaustion})
such that $g_i$ converge smoothly to $g_\infty$ on each $W_j$.
If 
\be \label{lem-edge-volume}
\vol_{g_i}(M\setminus W_j) \le V_j \textrm{ where } \lim_{j\to\infty}V_j=0
\ee
then there exists a uniform $V_0>0$ such that
\be \label{lem-vol}
\vol_{g_i}(M) < V_0.
\ee
\end{lem}

\begin{proof}
Fix any $W_j$.   Since $g_i$ converges smoothly on $W_j$,
$\vol_{g_i}(W_j)$ must converge smoothly as well.  So there
exists $V_1>0$ such that $\vol_{g_i}(W_j) \le V_1$.   Thus we
have
\be
\vol_{g_i}(M) =\vol_{g_i}(W_j) + \vol_{g_i}(M\setminus W_j)\le V_1+V_j
\ee
and $\sup V_j < \infty$ because $\lim_{j\to\infty} V_j$ exists.
\end{proof}

We now prove \sout{Proposition~\ref{flat-exhaustion}}
\textcolor{blue}{Theorem~\ref{flat-to-settled}}:

\begin{proof}
By hypothesis (\ref{not-vol-3}) and Lemma~\ref{lem-vol-vol} we have: 
\be \label{vol-3}
\vol(M_i) \le V_0,
\ee
Next we prove that $(W_j, g_\infty)$ satisfy the hypothesis of
\textcolor{blue}{Proposition~\ref{flat-exhaustion}}.   Observe that
hypothesis \textcolor{blue}{ (\ref{not-vol-3}) and smooth convergence we have}
\be
\vol_{g_\infty}(W_j)=\lim_{i\to\infty} \vol_{g_i}(W_j) \le V_0,
\ee
while (\ref{area-3}) implies
\be
\vol_{g_\infty}(\partial W_j) = \lim_{i\to\infty}
\vol_{g_i}(\partial W_j) \le A_0.
\ee
Finally 
\begin{eqnarray}
\diam_N(N) &=& \lim_{j\to \infty} \diam_N(W_j) \\
&\le& \lim_{j\to \infty} \lim_{k\to \infty} \diam_{(W_k, g_\infty)}(W_j) \\
&\le& \lim_{j\to \infty} \lim_{k\to \infty} \lim_{i\to\infty} \diam_{(W_k,g_i)}(W_j) \\
&\le& \lim_{j\to \infty} \lim_{k\to \infty} \lim_{i\to\infty} \diam_{(M,g_i)}(W_j)
+\lambda_{i,j,k} \\
&\le& \limsup_{j\to \infty} \limsup_{k\to \infty} \limsup_{i\to\infty} D_0+\lambda_{i,j,k} \\
&\le & D_0 +\lambda_0.   
\end{eqnarray}

Thus by \textcolor{blue}{Proposition~\ref{flat-exhaustion}}
\sout{Theorem~\ref{flat-to-settled}} we have
\be \label{settled-8}
d_{\mathcal{F}}\big((W_j,g_\infty), (N',d_\infty)\big)=F_j
\textrm{ where } \lim_{j\to \infty} F_j=0.
\ee

Next we will apply Theorem~\ref{thm-subdiffeo} to
show $M_1=(W_k, g_\infty)$ and $M_2=(M, g_i)$
are close in the intrinsic flat sense by setting
$U_1=W_j\subset W_k$ and $U_2=W_j \subset M$
for some well chosen $j<k$ 
Then
the values in the hypothesis of the theorem are
\begin{eqnarray}
\epsilon &=& \epsilon_{i,j} \textrm{ where } \lim_{i\to \infty} \epsilon_{i,j}=0,\\
D_{U_2}&\le& \diam_{(M,g_i)}(W_j) \le D_0\\
D_{U_1}&\le& \diam_{(W_k,g_i)}(W_j) \le D_0+\lambda_0\\
a&=&a_{i,j}\le a_{i,j}=2(D_0+\lambda_0)\arccos(1+\epsilon_{i,j})^{-1}/\pi \\
\lambda&=&\textcolor{blue}{\lambda'_{i,j,k} \textrm{ instead of } \lambda_{i,j,k}}  \label{lambda-prime}\\
h&=&h_{i,j,k}\le \sqrt{ \lambda'_{i,j,k}(D_0+\lambda_0+\lambda'_{i,j,k}/4)\,}\\
\bar{h}&=&\bar{h}_{i,j,k} \le \max\{h_{i,j,k},
\sqrt{\epsilon_{i,j}^2 + 2\epsilon_{i,j}}
(D_0+\lambda_0)\}
\end{eqnarray}
Thus
\be \label{Brian-Allen-concern}
d_{\mathcal{F}}\big((W_k,g_\infty),(M,g_i)\big)
\le
\left(\bar{h}_{i,j,k} + a_{i,j}\right) \Big(2 V_0+ 2 A_0\Big)+ 2V_j.
\ee
\textcolor{blue}{Brian Allen observed the above estimate was incorrect in the published version
because in (\ref{lambda-prime}) we had
$$
\lambda_{i,j,k}= \sup_{x,y\in W_j} |d_{(W_{k}, g_i)}(x,y)- d_{(M,g_i)}(x,y)|   
\textrm{ as in (\ref{lambda-ijk-5000})}
$$ 
but to apply Theorem~\ref{thm-subdiffeo} we need
$$
\lambda'_{i,j,k}=\sup_{x,y\in W_j} |d_{(W_{k}, g_\infty)}(x,y)- d_{(M,g_i)}(x,y)|. 
$$
We observe now that
$$
|\lambda'_{i,j,k}-\lambda_{i,j,k}| \le \eta_{i,j,k}
$$
where
$$
\eta_{i,j,k}= \sup_{x,y\in W_j} |d_{(W_{k}, g_i)}(x,y)-d_{(W_{k}, g_\infty)}(x,y)|.
$$
So by the smooth convergence of $g_i$ to $g_\infty$ on $W_k$ we have
$$
(1+\epsilon_{i,k})^{-2} g_\infty \le g_i \le (1+\epsilon_{i,k})^2 g_\infty \textrm{ on } W_k \textrm{ where } 
\lim_{i\to \infty} \epsilon_{i,k}=0
$$
Thus for any curve, $C$, in $W_k$ we have
$$
(1+\epsilon_{i,k})^{-1} L_{g_\infty}(C) \le L_{g_i}(C)\le (1+\epsilon_{i,k}) L_{g_\infty}(C)
$$
Applying this to a $g_i$-minimizing curve  $C_i$ from $x$ to $y$ in $W_k$ we have
\begin{eqnarray*}
d_{(W_{k}, g_\infty)}(x,y) &\le& L_{g_\infty}(C_i)\le (1+\epsilon_{i,k}) L_{g_i}(C_i)\\
&=& (1+\epsilon_{i,k})d_{(W_{k}, g_i)}(x,y) \\
&\le & d_{(W_{k}, g_i)}(x,y)  +  \epsilon_{i,k} (D_0+\lambda_0)
\end{eqnarray*}
and 
applying this to a $g_\infty$-minimizing curve  $C_\infty$ from $x$ to $y$ in $W_k$ we have
\begin{eqnarray*}
d_{(W_{k}, g_i)}(x,y) &\le& L_{g_i}(C_\infty)\le (1+\epsilon_{i,k}) L_{g_\infty}(C_\infty)\\
&=& (1+\epsilon_{i,k})d_{(W_{k}, g_\infty)}(x,y)\\
&\le& d_{(W_{k}, g_\infty)}(x,y) + \epsilon_{i,k} (1+\epsilon_{i,k})(D_0+\lambda_0)\\
\end{eqnarray*}
because
$$
d_{(W_{k}, g_\infty)}(x,y) \le (1+\epsilon_{i,k})d_{(W_{k}, g_i)}(x,y)\le  (1+\epsilon_{i,k})(D_0+\lambda_0)
$$
Thus
$$
\eta_{i,j,k}\le \eta_{i,k}=\epsilon_{i,k}  (1+\epsilon_{i,k})(D_0+\lambda_0)(D_0+\lambda_0)
$$
and for fixed $k$,
$$
\lim_{i\to\infty} \eta_{i,k}=0.
$$
So
$$
\lim_{i\to\infty} \lambda'_{i,j,k}=\lim_{i\to \infty}\lambda_{i,j,k}.
$$
This leads to the reordering of the limits in our fixed definition of uniformly well embedded:
$$
\limsup_{j\to \infty}\limsup_{k\to \infty} \limsup_{i \to \infty} \lambda_{i,j,k}=0
$$
which will imply
$$
\limsup_{j\to \infty}\limsup_{k\to \infty} \limsup_{i \to \infty} \lambda'_{i,j,k}=0
$$
and thus
$$
\limsup_{j\to \infty}\limsup_{k\to \infty} \limsup_{i \to \infty} \bar{h}_{i,j,k}=0.
$$
}

Combining (\ref{Brian-Allen-concern}) with (\ref{settled-8}) we have for any $j<k$,
\be
d_{\mathcal{F}}\big((N,g_\infty),(M,g_i)\big)
\le
\left(\bar{h}_{i,j,k} + a_{i,j}\right) \Big(2 V_0+ 2 A_0\Big)+ 2V_j +F_j .
\ee

\textcolor{blue}{So now we should take $i \to \infty$ first.
Recall that for any fixed $j$, $\lim_{i \to \infty} \epsilon_{i,j}=0$,
thus $\lim_{i\to \infty} a_{i,j}=0$ as well. 
$$
\limsup_{i\to\infty} d_{\mathcal{F}}\big((N',g_\infty),(M,g_i)\big)
\le
\left(\bar{h}_{j,k} + 0 \right) \Big(2 V_0+ 2 A_0\Big)+ 2V_j +F_j + 0.
$$
where $\bar{h}_{j,k}=\limsup_{i\to \infty} \bar{h}_{i,j,k}$.
Next taking the limsup as $k \to \infty$ 
$$
\limsup_{i\to\infty} d_{\mathcal{F}}\big((N',g_\infty),(M,g_i)\big)
\le
\left(\bar{h}_{j} + 0 \right) \Big(2 V_0+ 2 A_0\Big)+0+0 .
$$
where $\bar{h}_{j}=\limsup_{k\to \infty} \bar{h}_{j,k}$.
Taking the limsup as $j \to \infty$ 
$$
\limsup_{i\to\infty} d_{\mathcal{F}}\big((N',g_\infty),(M,g_i)\big)
\le
\left(0 + 0 \right) \Big(2 V_0+ 2 A_0\Big)+ 0 +0=0.
$$
}
\end{proof}

\subsection{Codimension 2 Singular Sets \textcolor{blue}{has Errors}}

The following lemma combined with Theorem~\ref{flat-to-settled}
completes the proof of Theorem~\ref{codim-thm}.   

\begin{lem}\label{codim-2-lambda}
 Let $M$ be compact, \textcolor{blue}{$g_i\to g_\infty$ smoothly away from $S$ uniformly from below}
 where  $S$ is a closed
 submanifold of codimension $2$ and $\diam_{g_\infty}(M\setminus S)<\infty$ 
 \textcolor{blue}{$\diam_{g_i}(M)\le D_0$}
 then, 
any connected precompact exhaustion, $W_j$, 
of $M\setminus S$ is uniformly well embedded.   
\end{lem}

\textcolor{blue}{With the correction to Definition~\ref{well-embedded}
the original proof of this lemma is no longer correct.   We now prove this lemma
using the new definition of smooth convergence away from $S$ uniformly from below and the
adapted definition of uniformly well embedded.  The proof is similar to the
original proof but we must be careful to take the limits in the correct order.}

\begin{proof}
\textcolor{blue}{
Observe that 
\be
d_{(W_{k}, g_i)}(x,y)- d_{(M,g_i)}(x,y)\ge 0
\ee
because $W_k\subset M$ and so
\begin{eqnarray}
d_{(W_{k}, g_i)}(x,y) &=& \inf \{ L_{g_i}(C)\,|\,\, C:[0,1]\to W_k,\,C(0)=x,\,\,C(1)=y\}\\
&\ge & \inf \{ L_{g_i}(C)\,|\,\, C:[0,1]\to M,\,C(0)=x,\,\,C(1)=y\}\\
&=& d_{(M,g_i)}(x,y).
\end{eqnarray}
Thus 
\be
\lambda_{i,j,k}= \sup_{x,y\in W_j} \,\,d_{(W_{k}, g_i)}(x,y)- d_{(M,g_i)}(x,y) .  
\ee  
Since $\bar{W}_j$ is compact, there exists  $x_{i,j,k}, y_{i,j,k} \in \bar{W}_j$ achieving this supremum:
\be
\lambda_{i,j,k}=d_{(W_{k}, g_i)}(x_{i,j,k},y_{i,j,k})- d_{(M,g_i)}(x_{i,j,k},y_{i,j,k}).
\ee
Consider a subsequence $i' \to \infty$ such that
\be
\lim_{i'\to \infty} \lambda_{i',j,k}=\limsup_{i\to \infty} \lambda_{i,j,k}
\ee
and consider a further subsequence, also denoted $i'$ such that
\be
x_{i',j,k} \to x_{\infty,j,k} \textrm{ and } y_{i',j,k} \to y_{\infty, j, k} \in \bar{W}_j.
\ee
In particular, as $i'\to \infty$ for fixed $j,k$, we have
\be
d_{g_\infty, W_k}(x_{i',j,k}, x_{\infty,j,k}) \to 0 \textrm{ and } d_{g_\infty, W_k}(y_{i',j,k}, y_{\infty,j,k}) \to 0 
\ee
Since $g_i \to g_\infty$ on $W_k$ for fixed $k$, there exists $H_{i,j,k}>1$ such that
\be
H_{i,j,k}^{-1} \ge \frac{d_{g_i, W_k}(p,q)}{d_{g_\infty, W_k}(p,q)}\ge H_{i,j,k} \qquad \forall p,q\in W_j
\ee
where
\be
\lim_{i\to \infty} H_{i,j,k} =1 \textrm{ for fixed } j,k.
\ee
Thus as $i'\to \infty$ we have
\be
d_{g_i', W_k}(x_{i',j,k}, x_{\infty,j,k}) \le H_{i',j,k} \cdot d_{g_\infty, W_k}(x_{i',j,k}, x_{\infty,j,k}) \to 1 \cdot 0 =0
\ee
and
\be
d_{g_i', W_k}(y_{i',j,k}, y_{\infty,j,k}) \le H_{i',j,k} \cdot d_{g_\infty, W_k}(y_{i',j,k}, y_{\infty,j,k}) \to 1 \cdot 0 =0
\ee
Combining these with the triangle inequality we have
\be
 |d_{(W_{k}, g_{i'})}(x_{i',j,k},y_{i',j,k})- d_{(W_{k}, g_{i'})}(x_{\infty,j,k},y_{\infty,j,k})| \to 0.
 \ee
 Note in addition that
 \be
d_{g_{i'}, M}(p,q) \le 
d_{g_{i'}, W_k}(p,q) 
\ee
so as $i'\to \infty$ for fixed $j,k$ we have
 \be
d_{g_{i'}, M}(x_{i',j,k}, x_{\infty,j,k})  \to  0
\textrm{ and }
d_{g_{i'}, M}(y_{i',j,k}, y_{\infty,j,k}) \to 0.
\ee
Combining these with the triangle inequality we have
\be
 |d_{(M, g_{i'})}(x_{i',j,k},y_{i',j,k})- d_{(M, g_{i'})}(x_{\infty,j,k},y_{\infty,j,k})| \to 0.
 \ee
 Thus
 \be
\limsup_{i\to \infty} \lambda_{i,j,k}=\lim_{i'\to \infty} 
d_{(W_{k}, g_{i'})}(x_{\infty,j,k},y_{\infty,j,k})- d_{(M,g_{i'})}(x_{\infty,j,k},y_{\infty,j,k}).
\ee
Let $\gamma_{i',j,k}$ be a $g_{i'}$ minimizing geodesic in $M$ between
$x_{\infty,j,k}$ and $y_{\infty,j,k}$.  Since $S$ is
a submanifold of codimension $2$, for any $h_{i'}\in (0,D_0)$, we can find a curve
$C_{i',j,k}: [0,1]\to M\setminus S$ between these points such that
\be
|L_{g_{i'}}(C_{i',j,k}) - d_{M, g_{i'}}(x_{\infty,j,k}, y_{\infty,j,k})| < h_{i'}
\ee
by sliding $\gamma_{i',j,k}$ over slightly to avoid $S$.
By the \textcolor{blue}{new definition of smooth convergence away from $S$ uniformly from below}
we have 
\be
g_i\ge (1-\delta_i)^2 g_\infty \textrm{ on } M \setminus S.
\ee
Thus
\begin{eqnarray}
d_{M, g_{i'}}(x_{\infty,j,k}, y_{\infty,j,k}) &\ge&   (1-\delta_i) L_{g_\infty}(C_{i',j,k}) -  h_{i'}\\
&\ge&   (1-\delta_i) d_{(M\setminus S,g_\infty)}(x_{\infty,j,k}, y_{\infty,j,k}) -  h_{i'}.\\
\end{eqnarray}
Since we can choose $\lim_{i'\to \infty} h_{i'} =0$ and we have $\delta_i\to 0$, 
\be
\lim_{i'\to \infty} d_{M, g_{i'}}(x_{\infty,j,k}, y_{\infty,j,k}) 
\ge  d_{(M\setminus S,g_\infty)}(x_{\infty,j,k}, y_{\infty,j,k}).
\ee
Since $g_i \to g_\infty$ uniformly on $\bar{W}_k$, we also have
\be
\lim_{i'\to \infty} 
d_{(W_{k}, g_{i'})}(x_{\infty,j,k},y_{\infty,j,k})=
d_{(W_{k}, g_\infty)}(x_{\infty,j,k},y_{\infty,j,k}).
\ee
Combining these we have
 \be
\limsup_{i\to \infty} \lambda_{i,j,k}\le
d_{(W_{k}, g_\infty)}(x_{\infty,j,k},y_{\infty,j,k})
-  d_{(M\setminus S,g_\infty)}(x_{\infty,j,k}, y_{\infty,j,k}).
\ee
Now choose a subsequence $k'$ such that
\be
\limsup_{k\to \infty}\limsup_{i\to \infty} \lambda_{i,j,k}=\lim_{k'\to \infty}\limsup_{i\to \infty} \lambda_{i,j,k}
\ee
and choose a further subsequence $k'$ such that 
\be
x_{\infty,j,k} \to x_{\infty,j} \subset \bar{W_j} \textrm{ and } y_{\infty,j,k} \to y_{\infty,j} \subset \bar{W_j}
\ee
By the fact that $\bar{W_j}\subset W_k \subset M\setminus S$ and the triangle inequality,
 \be
\limsup_{k\to \infty} \limsup_{i\to \infty} \lambda_{i,j,k}\le
\limsup_{k'\to \infty} d_{(W_{k'}, g_\infty)}(x_{\infty,j},y_{\infty,j})
-  d_{(M\setminus S,g_\infty)}(x_{\infty,j}, y_{\infty,j}).
\ee
For any $\epsilon_{j}>0$ we have  a curve $C_{j}:[0,1]\to M\setminus S$ 
running from $C_j(0)=x_{\infty,j}$ to $C_j(1)=y_{\infty,j}$
such that
\be
L_{g_\infty}(C_{j})<d_{(M\setminus S,g_\infty)}(x_{\infty,j}, y_{\infty,j})+\epsilon_{j}.
\ee
Since $W_{k'}$ exhaust $M\setminus S$, for $k'$ sufficiently large depending on $j$ we 
have $C_j([0,1])\subset W_{k'}$, so
\be
d_{(W_{k'}, g_\infty)}(x_{\infty,j},y_{\infty,j}) \le L_{g_\infty}(C_{j}).
\ee
Thus
\be
\limsup_{k\to \infty} \limsup_{i\to \infty} \lambda_{i,j,k}\le \epsilon_{j}.
\ee
Finally we apply the fact that we can choose $\epsilon_j\to 0$ so that
\be
\limsup_{j\to \infty} \limsup_{k\to \infty} \limsup_{i\to \infty} \lambda_{i,j,k}\le \epsilon_{j}.
\ee
}
 \end{proof}

\section{Intrinsic Flat to GH convergence}\label{Sect-flat-to-GH}

There are occasions where one has volume controls as
in Theorem~\ref{flat-to-settled} but one would like to obtain a
Gromov-Hausdorff limit.     That is not always possible.
Example~\ref{ex-no-GH} has no Gromov-Hausdorff limit
despite satisfying the conditions of Theorem~\ref{flat-to-settled}.
In Example~\ref{ex-not-GH}
the Gromov-Hausdorff limits and intrinsic flat limits
do not agree. However the second author and Stefan Wenger
have shown in \cite{SorWen2}
that the Gromov-Hausdorff and intrinsic flat limits
agree when the sequence of manifolds has nonnegative
Ricci curvature or a uniform contractibility function:

\begin{defn}\label{defn-contractibility-function}
A function $\rho:[0,r_0]\to [0,\infty)$ is a contractibility
function for a manifold $M$ with metric $g$ if 
every ball $B_p(r)$ is contractible within $B_p(\rho(r))$.
\end{defn}

We review these results in the first subsection.

In the second subsection, we apply the results in
\cite{SorWen2} on sequences of
manifolds with a uniform contractibility function, proving
Theorem~\ref{c-smooth-to-GH} and Theorem~\ref{c-codim-thm}.  

In the third subsection we use additional properties of
of manifolds with Ricci curvature bounds to prove additional
theorems about Gromov-Hausdorff limits inspired by the
techniques in \cite{SorWen2}.
In particular we prove 
Theorem~\ref{Ricci-smooth-to-GH}
and  Theorem~\ref{Ricci-codim-thm}.

\subsection{Review of Convergence Theorems}

First recall that Gromov proved
a sequence of compact Riemannian manifolds has a
subsequence converging in the Gromov-Hausdorff sense
if there is a uniform bound on the number of disjoint balls
of radius $r$ that fit in the space \cite{Gromov-metric}.  This lead to two
compactness theorems:

\begin{thm}[Gromov] \cite{Gromov-metric}
A sequence of compact Riemannian manifolds, $(M_j,g_j)$,
such that $\diam(M_j) \le D$ and $Ricci_{M_j} \ge -H$, has
a subsequence converging in the Gromov-Hausdorff
sense to a metric space $(X,d)$. 
\end{thm}

\begin{thm}[Greene-Petersen] \cite{Greene-Petersen}
A sequence of compact Riemannian manifolds, $(M_j,g_j)$,
such that $\vol(M_j) \le V$ and such that
there is a uniform contractibility
function,  $\rho:[0,r_0]\to [0,\infty)$, for all the $M_j$, has
a subsequence converging in the Gromov-Hausdorff
sense to a metric space $(X,d)$. 
\end{thm}

See Definition~\ref{defn-contractibility-function}.

In \cite{SorWen1} the following theorems were proven
which can be applied to deduce information about the Gromov-Hausdorff
limit of a sequence.

\begin{thm}[Sormani-Wenger] \label{thm-sw-Ricci}
If a sequence of oriented compact Riemannian manifolds, $(M_j,g_j)$,
such that $\diam(M_j) \le D$ and $Ricci_{M_j} \ge 0$
and $vol(M_j) \ge V_0  > 0$ converges in the Gromov-Hausdorff
sense to $(X,d)$, then it converges in the intrinsic flat sense
to $(X,d,T)$  (c.f. Theorem 4.16 of \cite{SorWen2}).
\end{thm}

This theorem is conjectured to hold with uniform lower bounds
on Ricci curvature \cite{SorWen2}.

\begin{thm}[Sormani-Wenger] \label{thm-sw-contractible}
If a sequence of oriented compact Riemannian manifolds, $(M_j,g_j)$,
with a uniform linear contractibility
function, $\rho: [0,\infty)\to [0,\infty)$
and a uniform upper bound on volume, $\vol(M_j)\le V$,
converges in the Gromov-Hausdorff
sense to $(X,d)$, then it converges in the intrinsic flat sense
to $(X,d,T)$ (c.f. Theorem 4.14 of \cite{SorWen2}).
\end{thm}

Recall that, in general, the intrinsic flat limits and Gromov-Hausdorff limits
need not agree [Examples~\ref{ex-cusp} and~\ref{ex-not-GH}] because
intrinsic flat limits do not include points with $0$ density as in (\ref{eq-positive-density}).
In fact intrinsic flat limits may exist when Gromov-Hausdorff limits do not
[Example~\ref{ex-no-GH}].

\subsection{Sequences with Uniform Contractibility Functions}

Recall Definition~\ref{defn-contractibility-function}.
Here we apply the results in
\cite{SorWen2} on sequences of
manifolds with a uniform contractibility function, stating and proving
Theorem~\ref{c-smooth-to-GH} and Theorem~\ref{c-codim-thm}.  Recall 
Definitions~\ref{defn-smoothly-1}
and~\ref{well-embedded}.

\begin{thm}\label{c-codim-thm}
Let $M_i=(M,g_i)$ be a sequence of oriented compact Riemannian manifolds
with a uniform linear contractibility
function, $\rho$, which converges smoothly away from
a codimension two submanifold, $S$, \textcolor{blue}{uniformly from below}.
If there is a connected precompact exhaustion of
$M\setminus S$  as in (\ref{defn-precompact-exhaustion})
satisfying the volume conditions
\be \label{area-4}
\vol_{g_i}(\partial W_j) \le A_0
\ee
and
\be \label{not-vol-4}
\vol_{g_i}(M\setminus W_j) \le V_j \textrm{ where } \lim_{j\to\infty}V_j=0,
\ee
then
\be
\lim_{j\to \infty} d_{GH}(M_j, N)=0,
\ee
where $N$ is the settled and
metric completion of $(M\setminus S, g_\infty)$.\footnote{\textcolor{blue}{This theorem is now corrected to require the convergence to be uniformly from below.} }
\end{thm}

\begin{thm}\label{c-smooth-to-GH}
Let $M_i=(M,g_i)$ be a sequence of compact oriented Riemannian manifolds
with a uniform linear contractibility
function, $\rho$, which converges smoothly away from
a singular set, $S$.  If there is a uniformly well embedded 
connected precompact exhaustion of
$M\setminus S$  as in (\ref{defn-precompact-exhaustion})
satisfying the volume conditions (\ref{area-4}) and (\ref{not-vol-4})
then
\be
\lim_{j\to \infty} d_{GH}(M_j, N)=0,
\ee
where $N$ is the settled and
metric completion of $(M\setminus S, g_\infty)$.
\end{thm}

\begin{rmrk} \label{c-necessity}
Example~\ref{ex-not-GH} has no uniform linear
contractibility near the singular set and the
Gromov-Hausdorff limit does not agree with the
intrinsic flat limit.   Examples~\ref{ex-no-GH} 
and~\ref{ex-not-bounded}, also satisfy
all the conditions of Theorem~\ref{c-codim-thm}
and~\ref{c-smooth-to-GH} except the existence
of a uniform linear contractibility function.
They have no Gromov-Hausdorff limit.   

The excess volume bound in (\ref{not-vol-2}) is shown to
be necessary in Example~\ref{ex-to-hemisphere} and
Example~\ref{ex-cap-cyl}.
The codimension two
condition of Theorem~\ref{c-codim-thm}
and the uniform embeddedness hypothesis of
Theorem~\ref{c-smooth-to-GH}
are seen to be necessary in Example~\ref{ex-to-torus-square}.
We believe we
have an example proving the necessity of
the uniform bound on the boundary volumes, (\ref{area-2}), 
and discuss this in Remark~\ref{maybe-contr-area}.
\end{rmrk}

\begin{rmrk}
It would be interesting to see whether the requirement
that the contractibility function is linear is a necessary
condition.  One might consider adapting the Example
by Schul and Wenger in the appendix of \cite{SorWen1}
to prove this.
\end{rmrk}

\begin{proof}
By Lemma~\ref{lem-vol-vol}, we have
\be \label{vol-4}
\vol(M_i) \le V_0.
\ee
This combined with the uniform contractibility function
allows us to apply the Greene-Petersen Compactness Theorem.
In particular we have a uniform upper bound on diameter:
\be\label{diam-4}
\diam(M_i) \le D_0,
\ee
We may now apply Theorem~\ref{flat-to-settled} to obtain
\be
    \lim_{j\to \infty} d_\mathcal{F}(M_j, N') =0
\ee
We then apply Theorem~\ref{thm-sw-contractible}
to see that the flat limit and Gromov-Hausdorff
limits agree due to the existence of the uniform
linear contractibility function
and the fact that the volume
is bounded below uniformly by the smooth limit.
In particular the metric completion and the
settled completion agree.
\end{proof}

We now easily prove Theorem~\ref{c-codim-thm}:

\begin{proof}
This theorem follows from Theorem~\ref{c-smooth-to-GH}
combined with Lemma~\ref{codim-2-lambda}.
\end{proof}

\subsection{Ricci curvature bounded below}

In this subsection we use additional properties 
of manifolds with Ricci curvature bounds to prove additional
theorems about Gromov-Hausdorff limits inspired by the
techniques in \cite{SorWen2}.
In particular we prove 
Theorem~\ref{Ricci-smooth-to-GH} and Theorem~\ref{Ricci-codim-thm}.  
Recall 
Definitions~\ref{defn-smoothly-1},~\ref{defn-precompact-exhaustion}
and~\ref{well-embedded}.

\begin{thm}\label{Ricci-smooth-to-GH}
Let $M_i=(M,g_i)$ be a sequence of oriented compact Riemannian manifolds
with uniform lower Ricci curvature bounds, 
\be
\Ricci_{g_i}(V,V)\ge (n-1)H \, g_i(V,V) \qquad \forall V \in TM_i
\ee
which converges smoothly away from
a singular set, $S\subset M$.  If there is a uniformly well embedded 
connected precompact exhaustion of
$M\setminus S$  as in (\ref{defn-precompact-exhaustion})
satisfying the volume conditions,
(\ref{area-2}) and (\ref{not-vol-2}),
and diameter bound (\ref{diam-2}),
then
\be
\lim_{i\to \infty} d_{GH}(M_i, N)=0
\ee
where $N$ is the settled and
metric completion of $(M\setminus S, g_\infty)$.
\end{thm}

When $H=0$ this theorem is an immediate consequence of
Theorem~\ref{flat-to-settled}.  In fact we need no diameter 
assumption in that setting:

\begin{lem} \label{lem-Ricci-diam}
Suppose we have a sequence of compact manifolds, $M_i=(M, g_i)$
with nonnegative Ricci curvature and
\be 
\vol(M_i) \le V_0
\ee
converging smoothly
away from a singular set to $(M\setminus S, g_\infty)$
then
\be\label{diam-2a}
\diam_{M_i}(W_j) \le \diam(M_i) \le D_0 \qquad \forall i\ge j, 
\ee
\end{lem}

\begin{proof}
Suppose not.  Let
$p\in W \subset \bar{W}\subset M\setminus S$ where $W$ is precompact
and let $q_i \in M_j$ such that $d_i=d_i(p,q_i) \to \infty$.
By smooth convergence on $W$, there exists $r_0>0$ such that
$B_p(r_0)\subset M_j$ smoothly converge to a ball in a smooth
Riemannian manifold.  In particular $\vol_{g_i}(B_p(r_0))\ge V_1$.
Then, by the Bishop-Gromov Volume Comparison Theorem, we have 
\begin{eqnarray}
V_0&\ge& \vol(B_q(d_i-r_0)) \\
&\ge& \frac{(d_i-r_0)^m}{(d_i+r_0)^m -(d_i-r_0)^m} \vol(Ann_q(d_i-r_0, d_i+r_0)) \\
&\ge & \frac{(d_i-r_0)^m}{d_i^m -(d_i-r_0)^m} \vol(B_p(r_0)) \\
&\ge & \frac{(d_i-r_0)^m}{2md_i^{m-1}r_0} V_1 
\end{eqnarray}
which gives a contradiction as $d_i \to \infty$.
\end{proof}

The lemma does not hold for a uniform lower bound on Ricci curvature
which is negative, as can be seen by taking a sequence of manifolds
approaching a complete noncompact hyperbolic manifold with finite
volume.

The following proposition handles the more general lower bounds on
Ricci curvature not addressed in \cite{SorWen1}:

\begin{prop} \label{prop-improve-sw}
Let $M_i=(M,g_i)$ be a sequence of oriented compact Riemannian manifolds with a uniform lower bound on Ricci curvature.
Suppose there is a connected precompact exhaustion of
$M\setminus S$  as in (\ref{defn-precompact-exhaustion})
satisfying the volume conditions
\be 
 \vol_{g_i}(M\setminus W_j) \le V_j \textrm{ where } \lim_{j\to\infty}V_j=0,
\ee
\be \label{prop-vol-sw}
\vol(M_i) \le V_0
\ee
and 
\be
\diam_(M_i) \le D_0 \qquad \forall i\ge j.
\ee
If $M_i$ converge smoothly away 
away from $S$ to $N=(M\setminus S, g_\infty)$.
Suppose also that $(M, g_i)$ converge in the
intrinsic flat sense to $N'$ where $N'$ is the settled
completion of $(M\setminus S, g_\infty)$.  Then
\be
d_{GH}(M_j, \bar{N}) \to 0
\ee
and the metric completion satisfies, $\bar{N}=N'$.
\end{prop}

\begin{proof}
By Gromov's Compactness theorem, we know that
a subsequence of the Riemannian manifolds $M_i=(M, g_i)$
converge to a compact metric space $(Y,d)$.  
Thus a subsequence of the manifolds converges in the
intrinsic flat sense to an integral current space, $(X,d,T)$,
where $X\subset Y$ \cite{SorWen2} [Thm 3.20].

By Theorem~\ref{flat-to-settled} and the
fact that intrinsic flat limits are unique, we know that the
settled completion of $(M\setminus S, g_\infty)$
is $(X,d,T)$.  In particular one needs no subsequence
to obtain the flat limit.

In the case where the sequence of metrics has nonnegative
Ricci curvature, Theorem~\ref{thm-sw-Ricci} implies that
$X=Y$.  In particular the settled completion is the metric completion
and so the Gromov-Hausdorff limit is the metric completion
of $(M\setminus S, g_\infty)$ and no subsequence was needed.

When the sequence of manifolds has a negative uniform
lower bound on Ricci curvature, we may imitate the proof of
Theorem~\ref{thm-sw-Ricci} which appears in\cite{SorWen1}.  
We must show that every $y\in Y$
lies in the settled completion of $(M\setminus S, g_\infty)$.

First observe that by the smooth convergence of $g_i$
away from $S$, we know the volumes are uniformly
bounded below:
\be
\vol_i(M_i)\ge V_0.
\ee
Thus we can apply the noncollapsing
theory of Cheeger-Colding \cite{ChCo-PartI} to see that
after possibly taking another subsequence of $(M, g_i)$
we can control the volumes of the limit space's balls:
For all $y \in Y$, there exists $y_i \in M$ such that
\be
\lim_{i\to\infty} \vol_i(B_{y_i}(r))=\mathcal{H}_m(B_y(r))\ge V_0 (r/D_0)^m>0,
\ee
where $B_y(r)\subset Y$ and $\mathcal{H}_m$
is the $m$ dimensional volume.   In particular, for $i$
sufficiently large
\be
 \vol_i(B_{y_i}(r))\ge (V_0/2) (r/D_0)^m>0.
\ee

Now we choose $j$ sufficiently large (depending on $r$), so that
\be
V_j< (V_0/4) (r/D)^m.
\ee
Then
\begin{eqnarray}
\vol_i(W_j \cap B_{y_i}(r))&\ge & \vol_i(B_{y_i}(r)) - \vol_i(M\setminus W_j)\\
&>& (V_0/4) (r/D_0)^m>0.
\end{eqnarray}
Thus there exists
\be\label{choose-z}
z_{r,i}\in W_j\cap B_{y_i}(r) \, \subset \, M\setminus S.
\ee
and
\be
\vol_i(W_j\cap B_{z_{r,i}}(2r)) \ge \vol_i(W_j \cap B_{y_i}(r)).
\ee
Since  $z_{r,i}\subset W_j \subset \bar{W}_j$, a subsequence of
the $z_{r,i}$ converge to $z_{r,\infty}\in \bar{W}_j \subset W_{j+1}$.
Since $g_i$ converge smoothly to $g_\infty$ on $W_{j+1}$,
\be \label{smooth-volume}
\vol_i(W_j\cap B_{z_{r,i}}(2r)) \to
\vol_\infty(W_j\cap B_{z_{r,\infty}}(2r)).
\ee
Thus
\be
\vol_\infty(W_j\cap B_{z_{r,\infty}}(2r))\ge (V_0/4) (r/D_0)^m>0
\ee
Note that by (\ref{choose-z}), taking the Gromov-Hausdorff
limit we see that
\be
d(z_{r,\infty}, y) < r.
\ee
So $y$ is in the metric completion of $(M\setminus S, g_\infty)$.
Furthermore
\begin{eqnarray}
\vol(B_y(3r)\cap (M\setminus S))
&\ge &\vol(B_y(3r)\cap W_j)\\
&\ge &\vol(B_{z_{r,\infty}}(2r)\cap W_j)\\
&\ge& (V_0/4) (r/D_0)^m>0
\end{eqnarray}
so $y$ is in the settled completion of $(M\setminus S, g_\infty)$.

In particular the settled completion is the metric completion
and so the Gromov-Hausdorff limit is the metric completion
of $(M\setminus S, g_\infty)$ and no subsequence was needed.
\end{proof}

We may now prove 
Theorem~\ref{Ricci-smooth-to-GH}:

\begin{proof}
The hypothesis (\ref{diam-2}),
 (\ref{area-2}) and (\ref{not-vol-2}),
allow us to apply Theorem~\ref{flat-to-settled}.   So 
$(M_i, g_i)$ has an intrinsic flat limit and that this intrinsic
flat limit is the settled completion of $(M\setminus S, g_\infty)$.
By Lemma~\ref{lem-vol-vol}, we
have (\ref{prop-vol-sw}).   Thus by
Proposition~\ref{prop-improve-sw}, the Gromov-Hausdorff
and Intrinsic Flat limits agree. 
\end{proof}

We now prove Theorem~\ref{Ricci-codim-thm} which
was stated in the introduction:

\begin{proof}
This theorem follows from Theorem~\ref{Ricci-smooth-to-GH}
combined with Lemma~\ref{codim-2-lambda}.
\end{proof}

\begin{rmrk} \label{Ricci-necessity}
The Ricci curvature condition is necessary in Theorem~\ref{Ricci-codim-thm}
as can be seen in
Example~\ref{ex-not-GH} and in Example~\ref{ex-no-GH}, which has no 
Gromov-Hausdorff limit.   
The excess volume bound in (\ref{not-vol-2}) is shown to
be necessary in Example~\ref{ex-to-hemisphere}.
All these examples satisfy the uniform embeddedness hypothesis
of Theorem~\ref{Ricci-smooth-to-GH} and demonstrate the necessity of these
conditions in that theorem as well.   
By Lemma~\ref{lem-Ricci-diam}, the diameter hypothesis
is not necessary when the Ricci curvature is nonnegative
although the volume condition is still necessary as seen in
Example~\ref{ex-cap-cyl}.
Otherwise we see this is a necessary condition in Example~\ref{ex-diam-now}.
We were unable to find an example proving the necessity of
the uniform bound on the boundary volumes, (\ref{area-2}), 
and suggest this as an open question in Remark~\ref{no-ex-Ricci-area}.
The codimension two condition of Theorem~\ref{Ricci-codim-thm}
and the 
uniform embeddedness hypothesis of Theorem~\ref{Ricci-smooth-to-GH}
are seen to be necessary for their respective theorems
in Example~\ref{ex-to-torus-square}.  
\end{rmrk}

\section{Appendix: Example of Brian Allen}

Brian Allen sketched out this example to the second author and we have filled in the details.   This example is highly technical and understanding the convergence requires modern methods .

\begin{example}
Let $g_0$ the standard flat metric on $M={\mathbb S}^1 \times {\mathbb S}^1 \times {\mathbb S}^1$.
Let 
\be
S={\mathbb S}^1 \times \{0\}\times \{0\} \subset M
\ee
which is a submanifold of codimension $2$.
Let $r: M \to [0, \infty)$ be the distance function from $S$:
\be
r(x)=\min\{ d_{g_0}(x, y):\, y\in S\}
\ee
and let $h_i:[0, \infty) \to [1/2, 1]$ be a smooth nonincreasing function which satisfies
\be
h_i(r) = 1/2  \textrm{ for } r\le 1/i  \textrm{ and } h_i(r) = 1  \textrm{ for } r\ge 2/i.
\ee
Taking
\be
g_i = h_i(r(x))^2 g_0
\ee
we have a sequence of Riemannian metrics on $M$ such that $g_i \to g_0$ smoothly
on compact sets in $M\setminus S$.  \footnote{\textcolor{blue}{Since
$\sup_{x \in M\setminus S} h_i(r(x)) -1 = 1/2$ for all $ i \in {\mathbb{N}}$,
we see that $g_i$ does not converge to $g_0$ on $M\setminus S$ uniformly from below.}}

We claim that
\be \label{Brian-Allen-metric}
\textrm{ the metric completion of }(M\setminus S, d_{g_0}) \textrm{ is isometric to }
(M, d_{g_0}).
\ee
 This can be seen since any geodesics in $(M, d_{g_0})$ can be approximated
by curves in $(M\setminus S, d_{g_0})$ that are arbitrarily close in length since $S$ has codimension 2.
Observe however that by the triangle inequality,
\be
d_{g_i} (p,q) \le d_{g_i}(p,p') + d_{g_i}(p',q') + d_{g_i}(q',q),
\ee
Since $g_i \le g_0$ everywhere and $g_i=(1/2)^2 g_0$ on $S$
and $S$ is a convex set with respect to $g_0$, we have
\be
d_{g_i} (p,q) \le d_\infty(p,q)
\ee
where
\be \label{Brian-1}
d_\infty(p,q)=\min\{ d_{g_0}(p,q), \, d_{g_0}(p,p') + (1/2)d_{g_0}(p',q') + d_{g_0}(q',q):\, p',q'\in S\}.
\ee
On the other hand we claim
\be
d_\infty(p,q) \ge d_{g_i}(p,q) - 3/j.
\ee
To see this take $C_i$ a $g_i$-minimizing geodesic
from $p$ to $q$, and take $p_i$ the first point on $C_i$ where it enters $r^{-1}[0, 1/i]$
and $q_i$ to be the last point in that set.  Then since $g_i \ge (1/2)^2g_0$
on $r^{-1}[0, 1/j]$ and $g_i=g_0$ elsewhere we have
\begin{eqnarray}
d_{g_i}(p,q)&=& d_{g_i}(p,p_j)+d_{g_i}(p_i,q_i) + d_{g_i}(q_i,q)\\
&\ge &d_{g_0}(p,p_i)+(1/2) d_{g_0}(p_i,q_i) + d_{g_0}(q_i,q)
\end{eqnarray}
Taking $p'_i, q'_j \in S$ closest to $p_i,q_i$ respectively, we know
\be
d_{g_0}(p_i',p_i) \le 1/i \textrm{ and } d_{g_0}(q_i',q_i) \le 1/i .
\ee
So
\begin{eqnarray}
d_{g_0}(p,p_i) &\ge& d_{g_0}(p,p'_i) -1/i\\
d_{g_0}(p_i,q_i) &\ge& d_{g_0}(p'_i,q'_i) - 2/i\\
d_{g_0}(q_i,q) &\ge& d_{g_0}(q,q'_i) - 1/i
\end{eqnarray}
Thus we have our claim because
\begin{eqnarray}
d_{g_i}(p,q)
&\ge &d_{g_0}(p,p'_i)+(1/2) d_{g_0}(p'_i,q'_i) + d_{g_0}(q'_i,q) + 3/i \\
&\ge& d_\infty(p,q).
\end{eqnarray}
So in fact we have $d_i$ converges pointwise to $d_\infty$.   Following the arguments
in the first two papers of Allen-Sormani applying the Appendix to Huang-Lee-Sormani
and the fact that 
\be
(1/2) d_{g_0}(p,q) \le d{g_i}(p,q)\le d_{g_0}(p,q)
\ee
we get uniform, intrinsic flat, and Gromov-Hausdorff convergence of
\be
(M, d_{g_i}) \to (M, d_\infty)
\ee
which according to (\ref{Brian-Allen-metric}) is not the metric completion of $(M\setminus S, g_0)$
even though $g_i\to g_0$ on compact sets away from $S$.
\end{example}

\begin{rmrk}
This example is a counter example to the original statement of Theorem~\ref{codim-thm} because 
$M_i=(M,g_i)$ is a sequence of compact oriented Riemannian manifolds
such that $S$ is a codimension $2$ submanifold
and we can choose a connected precompact exhaustion,
\be
W_j = r^{-1}[2/j, \infty) \subset M \setminus S
\ee
satisfying (\ref{defn-precompact-exhaustion})
\be 
\bar{W}_j \subset W_{j+1} \textrm{ with } 
\bigcup_{j=1}^\infty W_j=M\setminus S
\ee
with $g_i$ converge smoothly to $g_0$ on each $W_j$, in fact $g_i=g_0$ for $i>j$.
Furthermore
\be
\diam_{M_i}(W_j) \le \diam_{g_0}(M)=D_0 \qquad \forall i\ge j, 
\ee
\be \label{m-area}
\vol_{g_i}(\partial W_j) \le \vol_{g_i}(\partial W_j)=4\pi (2/j) \pi
\ee
and
\be \label{m-edge-volume}
\vol_{g_i}(M\setminus W_j) \le \vol_{g_0}(M\setminus W_j) = (4/3)\pi (2/j)^2 \pi =V_j \textrm{ where } \lim_{j\to\infty}V_j=0.
\ee
However
\be
\lim_{j\to \infty} d_{\mathcal{F}}(M_j', N')=0.
\ee
where  $N'$ is the settled completion
of $(M\setminus S, g_0)$. 
\end{rmrk}

\begin{rmrk}
This example is not a counter example to Theorem~\ref{Ricci-codim-thm} because
of the highly negative sectional and Ricci curvature near $S$.   
\end{rmrk}

\begin{rmrk}
This example is not a counter example to Theorem~\ref{flat-to-settled} because $M/S$
is not uniformly well embedded as defined in the new Definition~\ref{well-embedded}.  Consider
a pair of points $p,q\in M \setminus S$ and $p',q'\in S$ such that
\be
d_\infty(p,q) =d_{g_0}(p,p_i)+(1/2) d_{g_0}(p',q') + d_{g_0}(q_i,q)
< d_{g_0}(p,q).
\ee
Taking any connected precompact exhaustion $W_j$ of $U=M\setminus S$, we
can take $j>k$ sufficiently large that $p,q\in W_j \subset W_k$.
We can take $i$ sufficiently large depending on $j>k$ such that 
\be
W_j \cap r^{-1}[0, 1/(2i_{j,k})] = \emptyset.
\ee
Then
\begin{eqnarray}
\lambda_{i,j,k}&=& \sup_{x,y\in W_j} |d_{(W_{k}, g_i)}(x,y)- d_{(M,g_i)}(x,y)|  \\ 
&\ge & |d_{(W_{k}, g_i)}(p,q)- d_{(M,g_i)}(p,q) | \\
&\ge &  d_{g_0}(p,q)- d_{g_i}(p,q).  
\end{eqnarray}
By the pointwise convergence proven in the example we have
\be
\limsup_{i\to \infty} \lambda_{i,j,k} = d_{g_0}(p,q) - d_\infty(p,q)
\ee
so 
\be
\limsup_{j\to \infty}\limsup_{k\to \infty} \limsup_{i \to \infty} \lambda_{i,j,k}
\ge d_{g_0}(p,q) - d_\infty(p,q)>
0
\ee
and we fail to satisfy (\ref{lambda-ijk-2-2020}).
\end{rmrk}

\bibliographystyle{AMSalpha}
\bibliography{2011}

\end{document}